\documentclass[]{scrartcl}
\usepackage[letterpaper,bindingoffset=0cm,inner=2.5cm,outer=2.5cm,top=2.5cm,bottom=2.5cm]{geometry}
\usepackage{helvet}
\usepackage[T1]{fontenc}

\usepackage{amssymb, amsthm}
\usepackage{fontawesome5}
\usepackage{bm}
\usepackage{mathtools}
\usepackage{amsmath,amsfonts}
\usepackage{xcolor}
\usepackage{xspace}
\usepackage{xstring}
\usepackage{subcaption}
\usepackage{booktabs}
\usepackage{tabularx}
\usepackage{numprint}
\usepackage{gensymb}
\npdecimalsign{.}
\usepackage{longtable}
\usepackage{import}
\usepackage{pgfplots}
\usepackage{pgfplotstable}
\usepackage{tcolorbox}
\usepackage{float}

\usepackage{appendix}

\usepackage[colorlinks=true]{hyperref}
\definecolor{mylinkcolor}{RGB}{0,0,130}
\hypersetup{colorlinks,allcolors=mylinkcolor,citecolor=mylinkcolor}
\usepackage{url}
\usepackage[backend=biber,%
                      bibencoding=utf8,%
                      giveninits=true,%
                      isbn=false,%
                      sortcites=true,%
                      natbib=true,%
                      style=ext-numeric,
                ]{biblatex}
\DeclareFieldFormat
  [article,inbook,incollection,inproceedings,patent,thesis,unpublished]
  {titlecase:title}{\MakeSentenceCase*{#1}}

\usepackage[scr=boondoxo,scrscaled=1.05,frak=mma,frakscaled=.97]{mathalfa}

\usepackage[capitalize,nameinlink]{cleveref}
\crefformat{equation}{(#2#1#3)}

\usepackage[font=small,labelfont=bf]{caption}

\providecommand{\keywords}[1]{\textbf{\textit{Keywords ---}} #1}
\providecommand{\msc}[1]{\textbf{\textit{MSC 2010 ---}} #1}
\providecommand{\acknowledgements}[1]{\textbf{\textit{Acknowledgements ---}} #1}

\newtheorem{mydef}{Definition}

\newtheorem{mythm}{Theorem}

\crefname{secname}{Section}{Section}
\Crefname{secname}{Sec.}{Sec.}
\crefname{section}{Section}{Sections}
\crefname{subsection}{Subsection}{Subsections}

\usepackage{tikz}

\usepackage{bm}
\usepackage{amsmath}
\usepackage{xcolor}
\usepackage{pgfplots}
\pgfplotsset{compat=newest}
\usetikzlibrary{matrix, calc, positioning, fit, patterns, decorations.pathmorphing, decorations.markings, positioning, angles, shapes.geometric, math, arrows.meta, fadings, through, quotes}

\definecolor{phColor}{RGB}{255, 131, 131}
\definecolor{inputColor}{rgb}{1.0, 0.75, 0.0}

\colorlet{weightColor}{phColor}

\definecolor{stateColor}{RGB}{0,104,163}
\definecolor{pcaColor}{RGB}{60, 116, 200}
\definecolor{intermediateLatentColor}{RGB}{120, 122, 230}
\definecolor{aeColor}{RGB}{185, 122, 249}
\definecolor{latentColor}{RGB}{252, 114, 255}

\definecolor{plotColor1}{rgb}{0.4, .4, .4}
\definecolor{plotColor2}{rgb}{1, 0.0, 1.0}
\definecolor{plotColor3}{rgb}{0.5, 0.0, 0.5}
\definecolor{plotColor4}{rgb}{0, 1.0, 1.0}

\tikzset{
        cell/.style={
                rectangle, 
                rounded corners=.25mm, 
                minimum height =0.7cm, 
                minimum width=0.7cm, 
                draw=gray,
                anchor=center,
        },
        trapez/.style={
                trapezium, 
                rounded corners=0.15mm,
                minimum height=12mm, 
                trapezium stretches=true,
                trapezium angle=80, 
                inner sep=0mm,
                draw=gray,
                shape border rotate=270, 
        },
        mylabel/.style={%
                align=center,
                text=darkgray,
        },      
        Arrow/.style={
                -stealth,
                rounded corners=.1cm,
                thick,
                draw=gray,
                shorten >= 3pt,
                shorten <= 3pt
        },
        neuron/.style={circle, draw, minimum size=3ex, thick},
        operator/.style={circle, minimum width=7mm, fill=lightgray!40!white, draw=lightgray},
        equation/.style={cell},
}

\pgfdeclarelayer{bg0}    
\pgfdeclarelayer{bg1}    
\pgfsetlayers{bg1, bg0, main}

\newcommand\drawNodes[3]{
  \foreach \neurons/\color [count=\lyrIdx] in #2 {
    \StrCount{\neurons}{,}[\arrlength] 
    \foreach \n [count=\nIdx] in \neurons
      \node[neuron, fill=\color!40!white, draw=\color!80!black] (#1-\lyrIdx-\nIdx) at (1.75*\lyrIdx, -2*\nIdx) {\n};
  }
}

\newcommand\denselyConnectNodes[3]{

  \begin{pgfonlayer}{bg0}
  \foreach \n [count=\lyrIdx, remember=\lyrIdx as \previdx, remember=\n as \prevn] in #2 {
    \foreach \y in {1,...,\n} {
      \ifnum \lyrIdx > 1
        \foreach \x in {1,...,\prevn}
          \draw[-, draw=#3, 
          densely dashed,
          thick] (#1-\previdx-\x.center) -- (#1-\lyrIdx-\y.center);
      \fi
    }
  }
  \end{pgfonlayer}
}

\author{Johannes Rettberg\thanks{Institute of Engineering and Computational Mechanics, University of Stuttgart, Pfaffenwaldring 9, 70569 Stuttgart, Germany. (\url{johannes.rettberg,jonas.kneifl,joerg.fehr@itm.uni-stuttgart.de})} 
\and Jonas Kneifl\footnotemark[1] 
\and Julius Herb\thanks{Institute of Applied Mechanics – Data Analytics in Engineering, University of Stuttgart, Universitätsstr.
32, 70569 Stuttgart, Germany. (\url{julius.herb@mib.uni-stuttgart.de})}
\and Patrick Buchfink\thanks{Department of Applied Mathematics, University of Twente, P.O. Box~217, 7500~AE Enschede. (\url{p.buchfink@utwente.nl})} 
\and J\"org Fehr\footnotemark[1] 
\and Bernard Haasdonk\thanks{Institute of Applied Analysis and Numerical Simulation, University of Stuttgart, Pfaffenwaldring 57, 70569 Stuttgart, Germany. (\url{haasdonk@mathematik.uni-stuttgart.de})}}
\title{Data-driven identification of latent port-Hamiltonian systems}
\graphicspath{{fig/disc_brake/3d/data}}

\begin{document}
\maketitle
\newcommand{\todo}[1]{\textcolor{yellow}{\textbf{Todo:} #1}}
\newcommand{\johannes}[1]{\textcolor{magenta}{\textbf{JR:}~#1}}
\newcommand{\jonas}[1]{\textcolor{cyan}{\textbf{JK:}~#1}}
\newcommand{\patrick}[1]{\textcolor{orange}{\textbf{PB:}~#1}}
\newcommand{\julius}[1]{\textcolor{green}{\textbf{JH:}~#1}}
\newcommand{\bernard}[1]{\textcolor{purple}{\textbf{BH:}~#1}}
\newcommand{\joerg}[1]{\textcolor{teal}{\textbf{JF:}~#1}}

\newcommand{\phin}{pHIN} 
\newcommand{\aphin}{ApHIN} 

\newcommand{\Rdim}{\ensuremath\mathbb{R}}
\newcommand{\N}{\ensuremath\mathbb{N}}
\newcommand{\contFun}{\ensuremath\mathcal{C}}
\newcommand{\transpose}{\intercal}
\newcommand{\rT}[1]{#1^\transpose}
\newcommand{\rTb}[1]{\left(#1\right)^\transpose}
\newcommand{\rTsb}[1]{\left[#1\right]^\transpose}
\newcommand{\eye}{\bm{I}}
\newcommand{\orderOf}{\mathcal{O}}
\newcommand{\inv}[1]{#1^{-1}}
\newcommand{\diag}{\text{diag}}
\newcommand{\rmd}{\mathrm{d}}
\newcommand{\ddt}{\frac{\mathrm{d}}{\mathrm{d}t}}
\newcommand{\tddt}{\tfrac{\mathrm{d}}{\mathrm{d}t}}
\newcommand{\binare}{\ensuremath\{0,1\}}
\newcommand{\norm}[1]{\left\lVert#1\right\rVert}
\newcommand{\abs}[1]{\left|#1\right|}

\newcommand{\approximate}[1]{\widetilde{#1}}
\newcommand{\reconstruct}[1]{\breve{#1}}

\newcommand{\Flow}{\bm{F}}
\newcommand{\rhs}{\bm{f}}
\newcommand{\rhsState}{\rhs_{\stateSpace}}
\newcommand{\rhsLatentState}{\rhs_{\latentStateSpace}}
\newcommand{\rhsLatentStatePH}{\rhs_{\latentStateSpace,\mathrm{pH}}}

\newcommand{\states}{\bm{x}}
\newcommand{\statesTraj}[1][t;\params]{\states(#1)}
\newcommand{\statesDotTraj}[1][t;\params]{\dot{\states}(#1)}
\newcommand{\statesPoint}{\bm{\mathscr{{\states}}}}
\newcommand{\stateSpace}{\mathcal{X}}
\newcommand{\stateDim}{N}

\newcommand{\statesApprox}{\approximate{\states}}
\newcommand{\statesApproxTraj}[1][t;\params]{\statesApprox(#1)}
\newcommand{\statesApproxPoint}{\approximate{\statesPoint}}
\newcommand{\statesApproxDotTraj}[1][t;\params]{\dot{\statesApprox}(#1)}

\newcommand{\latentStateDim}{r}
\newcommand{\latentState}{\bm{z}}
\newcommand{\latentStatePoint}{\bm{\mathscr{{\latentState}}}}
\newcommand{\latentStateTraj}[1][t;\params]{\latentState(#1)}
\newcommand{\latentStateDotTraj}[1][t;\params]{\dot{\latentState}(#1)}
\newcommand{\latentStateSpace}{\mathcal{Z}}
\newcommand{\intermediateState}{\bm{{v}}}
\newcommand{\intermediateStateDim}{n_{\intermediateState}}

\newcommand{\param}{\mu}
\newcommand{\params}{\bm{\param}}
\newcommand{\paramSpace}{\mathcal{P}}
\newcommand{\paramDim}{n_{\param}}
\newcommand{\timee}{t}
\newcommand{\timeSpace}{\mathcal{T}}
\newcommand{\timeFinal}{\timee_{\mathrm{end}}}
\newcommand{\nTime}{\eta}
\newcommand{\inputU}{u}
\newcommand{\inputs}{\bm{\inputU}}
\newcommand{\inputSpace}{\mathcal{U}}
\newcommand{\residual}{\bm{r}}

\newcommand{\funInput}{\bm{g}}
\newcommand{\funOutput}{\bm{h}}

\newcommand{\A}{\bm{A}}
\newcommand{\M}{\bm{M}}
\newcommand{\ltiAid}{\approximate{\A}}
\newcommand{\ltiBid}{\phBid}
\newcommand{\phJ}{\bm{J}}
\newcommand{\phR}{\bm{R}}
\newcommand{\phQ}{\bm{Q}}
\newcommand{\phDescriptor}{\bm{E}}
\newcommand{\phB}{\bm{B}}
\newcommand{\Hamiltonian}{\mathcal{H}}
\newcommand{\phState}{\approximate{\latentState}}
\newcommand{\phStateTraj}[1][t;\params]{\approximate{\latentState}(#1)}
\newcommand{\phStatePoint}{\approximate{\bm{\mathscr{\latentState}}}}
\newcommand{\phStateDot}{\dot{\phState}}
\newcommand{\phStateDotTraj}[1][t;\params]{\dot{\phState}(#1)}
\newcommand{\phStateDotPoint}{\dot{\phStatePoint}}
\newcommand{\phSpace}{\latentStateSpace}
\newcommand{\phDim}{\latentStateDim}
\newcommand{\phPortDim}{n_{\mathrm{p}}}
\newcommand{\phInputU}{\inputs}
\newcommand{\phInputUTraj}[1][t]{\inputs(#1)}
\newcommand{\phInputUPoint}{\bm{\mathscr{{\inputs}}}}
\newcommand{\phInputSpace}{\inputSpace}
\newcommand{\phOutput}{\bm{y}}
\newcommand{\phOutputTraj}[1][t;\params]{\phOutput(#1)}
\newcommand{\phOutputPoint}{\bm{\mathscr{{\phOutput}}}}
\newcommand{\phOutputSpace}{\mathcal{Y}}

\newcommand{\phJid}{\approximate{\phJ}}
\newcommand{\phRid}{\approximate{\phR}}
\newcommand{\phQid}{\approximate{\phQ}}
\newcommand{\phBid}{\approximate{\phB}}

\newcommand{\nSamples}{n_t}
\newcommand{\nSamplesTest}{\nSamples^{\text{test}}}
\newcommand{\solutionManifold}{\mathcal{S}}

\newcommand{\disp}{q}
\newcommand{\disps}{\bm{\disp}}
\newcommand{\dispSpace}{\mathcal{Q}}
\newcommand{\temp}{\vartheta}
\newcommand{\temps}{\bm{\temp}}
\newcommand{\momentum}{p}
\newcommand{\momentums}{\bm{\momentum}}
\newcommand{\nNodes}{n_{\mathrm{nodes}}}
\newcommand{\stiffness}{k}
\newcommand{\mass}{m}
\newcommand{\damping}{c}
 \newcommand{\heatflux}{Q}
 \newcommand{\heatcond}{\lambda_{\mu}}
 \newcommand{\density}{\rho_{\mu}}

\newcommand{\eigenvalue}{\lambda}
\newcommand{\eigenvalueMatrix}{\bm{\Lambda}}
\newcommand{\singvalue}{\varsigma}
\newcommand{\singvalues}{\bm{\singvalue}}

\newcommand{\pcaRed}{\bm{V}^{\transpose}}
\newcommand{\pcaRec}{\bm{V}}

\newcommand{\neuralnet}{\bm{\Psi}}
\newcommand{\fullEncoder}{\neuralnet_{\text{e}}} 
\newcommand{\fullDecoder}{\neuralnet_{\text{d}}} 
\newcommand{\encoder}{\fullEncoder^*}
\newcommand{\decoder}{\fullDecoder^*}
\newcommand{\loss}{\mathcal{L}}
\newcommand{\lossfactor}{\lambda}
\newcommand{\dataSet}{\mathcal{D}}
\newcommand{\weight}{{\theta}}
\newcommand{\weights}{\bm{\weight}}
\newcommand{\Weights}{\bm{\Theta}}
\newcommand{\activationFunction}{h}
\newcommand{\nUnits}{n_\text{u}}
\newcommand{\nLayers}{n_\text{l}}
\newcommand{\numData}{n_{\mathrm{data}}}
\newcommand{\regPosDef}{\epsilon}
\newcommand{\weightsPH}{\weights_{\mathrm{pH}}}
\newcommand{\neuralnetWeights}{\neuralnet_{\mathrm{pH}}}
\newcommand{\weightsNeuralnetWeights}{\weights_{\neuralnet}}
\newcommand{\weightsAutoenc}{\weights_{\mathrm{AE}}}
\newcommand{\fullEncoderWithWeights}{\neuralnet_{\text{e,} \weightsAutoenc}}
\newcommand{\fullDecoderWithWeights}{\neuralnet_{\text{d,} \weightsAutoenc}}

\newcommand{\error}{\bm{e}}
\newcommand{\errorProj}{\error_{\text{proj}}}
\newcommand{\errorJac}{\error_{\text{jac}}}
\newcommand{\errorLatent}{\error_{\latentState}}
\newcommand{\errorState}{\error_{\states}}
\newcommand{\errorLatentMean}{\bar{\error}_{\latentState}}
\newcommand{\errorStateMean}{\bar{\error}_{\states}}

\newcommand{\pendulumLength}{l}
\newcommand{\pendulumAngle}{\varphi}
\newcommand{\pendulumAngVel}{\omega}
\newcommand{\pendulumPos}{\rho}
\newcommand{\gravitation}{g}
\newcommand{\nSims}{n_\text{s}}

\newcommand{\lipschitzConst}{L_{\fullDecoder, \overline{\boundedNeighbourhood}}}
\newcommand{\HamiltonianMin}{\Hamiltonian_{\text{min}}}
\newcommand{\boundedNeighbourhood}{U}
\newcommand{\boundedNeighbourhoodApprox}{\approximate\boundedNeighbourhood}
\newcommand{\timeBoundary}{t_{\partial}}
\newcommand{\constBoundedSolutions}{C_{\latentStateSpace}}
\newcommand{\funStorage}{S}
\newcommand{\stateEquilibrium}{\phStatePoint_{\mathrm{e}}}

\newcommand{\evalField}[2][\states]{#2|_{#1}}
\newcommand{\domainLyapunov}{W}
\newcommand{\domainLyapunovApprox}{\approximate\domainLyapunov}
\newcommand{\totalDiff}{D}
\newcommand{\JacState}{\totalDiff_{\states}}
\newcommand{\JacLatentState}{\totalDiff_{\latentState}}
\newcommand{\stateSpaceApprox}{\approximate{\stateSpace}}
\newcommand{\phStateApprox}{\approximate\states}
\newcommand{\phJApprox}{\overline\phJ}
\newcommand{\phRApprox}{\overline\phR}
\newcommand{\phBApprox}{\overline\phB}
\newcommand{\phOutputApprox}{\overline\phOutput}
\newcommand{\HamiltonianApprox}{\overline\Hamiltonian}
\newcommand{\funStorageApprox}{\overline\funStorage}
\newcommand{\stateEquilibriumApprox}{\statesApproxPoint_{\mathrm{e}}}

\begin{abstract}
    \small
    \textbf{Abstract}
    Conventional physics-based modeling techniques involve high effort, e.g.,~time and expert knowledge, while data-driven methods often lack interpretability, structure, and sometimes reliability. To mitigate this, we present a data-driven system identification framework that derives models in the port-Hamiltonian (pH) formulation.
    This formulation is suitable for multi-physical systems while guaranteeing the useful system theoretical properties of passivity and stability.

    Our framework combines linear and nonlinear reduction with structured, physics-motivated system identification.
    In this process, high-dimensional state data obtained from possibly nonlinear systems serves as input for an autoencoder, which then performs two tasks: (i) nonlinearly transforming and (ii) reducing this data onto a low-dimensional latent space.
    In this space, a linear pH system, that satisfies the pH properties per construction, is parameterized by the weights of a neural network. The mathematical requirements are met by defining the pH matrices through Cholesky factorizations.
    The neural networks that define the coordinate transformation and the pH system are identified in a joint optimization process to match the dynamics observed in the data while defining a linear pH system in the latent space.
    The learned, low-dimensional pH system can describe even nonlinear systems and is rapidly computable due to its small size.

    The method is exemplified by a parametric mass-spring-damper and a nonlinear pendulum example, as well as the high-dimensional model of a disc brake with linear thermoelastic behavior.
\end{abstract}

\msc{93B30, 37E99}


\section{Introduction}\label{sec:intro}
\emph{Mathematical models} and their numerical approximation allow us to perform simulations of complex phenomena which replace expensive or infeasible real-world experiments.
Moreover, a mathematical model enables more complex model-based tasks like simulations, optimization, control, or uncertainty quantification.

However, researchers are often faced with the challenges
(CH1-1) that deriving a mathematical and computational model requires expensive expert knowledge and a long development time, (CH1-2) that the underlying phenomena remain poorly understood and consequently violate fundamental physical constraints, or (CH2) that a model may be available but it is too computationally intensive for the intended use.
At the same time, sensors have become cheaper, and the amount of collected data, such as measurements or camera data, has increased.
Thus, current research tries to develop methods that derive models from such data, which we refer to in the following as \emph{(data-based) model discovery methods}.
These methods are relevant in \emph{system identification}, which develops methods to identify mathematical models from data to face challenges (CH1-1) and (CH1-2).
If the underlying models are very high-dimensional,
classical methods from system identification typically become computationally intractable.
In this case, \emph{non-intrusive model (order) reduction} (non-intrusive MOR) can be used.
Moreover, non-intrusive MOR offers methods to derive a computationally efficient surrogate model for the intractable models from challenge (CH2).
These techniques assume that, although the intractable model is high-dimensional{, its set of all solutions can be well-approximated with a low-dimensional manifold}, which is used to obtain a low-dimensional, computationally efficient model.

To improve the prediction capability of model discovery methods, prior knowledge can be incorporated, which is referred to as \emph{knowledge infusion}.
This is typically achieved by restricting the class of the discovered models to a set of models that guarantees desirable physical properties.
This can include laws of conservation, symmetries, or certain structures that should occur in the system.
The idea is that the discovered model is restricted from being able to behave physically, thus improving the prediction capability, especially in the regime of extrapolation. In particular, we are interested in the model discovery of \emph{port-Hamiltonian systems} through the \emph{concept of coordinates and dynamics}.

\emph{Port-Hamiltonian} (pH) systems are energy-based formulations that offer so-called ports to interconnect with controllers or other systems \cite{SchaftJeltsema14}, which is favorable for modular multi-physical modeling (see e.g.\ \cite{Rettberg23}).
The structure of pH systems allows the guarantee of desirable system-theoretical properties like passivity, stability, or boundedness of solutions under mild assumptions.
Moreover, port-Hamiltonian systems are closed under interconnection, i.e., connecting pH systems again results in a pH system, which is an attractive property for network modeling \cite{Mehrmann23}.

The idea of the \emph{concept of coordinates and dynamics} is that the complexity of the solution of a system is closely tied to the coordinate system in which it is described in~\cite{BruntonKutz2022}.
Accordingly, a prescribed coordinate system may not be suited for an efficient description of the underlying dynamics.
Hence, the concept of coordinates and dynamics simultaneously discovers a dynamical system along with a (nonlinear) coordinate transformation such that the dynamics in the transformed coordinates are simpler (e.g., in the sense that the transformed dynamics are parsimonious).
A typical approach to parameterize the nonlinear coordinate transformation are so-called autoencoders.
\emph{Autoencoder-based} model discovery methods have been developed, e.g., in \cite{Champion19,Bakarji2023}.
The advantage of autoencoders is that these can work with unlabelled training data and thus do not need time-intensive labeling.
Moreover, they can be used for nonlinear dimension reduction, which allows model discovery methods to be used even for high-dimensional data (e.g., from discretized partial differential equations (PDEs)).

In our work, we develop two model discovery methods that derive a pH system from data, namely
(i) the (parametric) \emph{port-Hamiltonian identification network}~(\phin{}) and
(ii) the \emph{autoencoder-based port-Hamiltonian identification network}~(\aphin{}).
While (i) is intended to identify a low-dimensional pH system directly,
(ii) makes use of an autoencoder-based dimension reduction and is designed to derive an intrinsically low-dimensional pH model from high-dimensional data.

\subsection{Main contributions}
Our three main contributions are:
\begin{enumerate}
    \item We develop two new model discovery methods, \phin{} and \aphin{},
          which are able to discover a port-Hamiltonian model from data and are relevant for
          \begin{itemize}
              \item data-driven system identification of port-Hamiltonian systems,
              \item structure-preserving, non-intrusive MOR of port-Hamiltonian systems, and
              \item structure-revealing, non-intrusive MOR of port-Hamiltonian systems.
          \end{itemize}
    \item For the \aphin{}, we prove that, despite the coordinate transformation with the autoencoder, the desirable system-theoretical properties like passivity, Lyapunov stability, and boundedness of solutions, are transferred from the low-dimensional into the original high-dimensional space.
    \item We numerically demonstrate our approach on the academic examples of a parametric mass-spring-damper system and a nonlinear pendulum and further show its capabilities for a high-dimensional model of a thermo-mechanical disc brake based on a finite element (FE) based discretized PDE.
\end{enumerate}

\subsection{State-of-the-art}
With the three different interpretations of our method listed in the previous subsection,
the new methods tackle multiple contemporary research topics at once.
In the following, we highlight related work.

\paragraph*{Data-driven System Identification (of Port-Hamiltonian Systems)}
The desire to discover models in an automated fashion directly from data has driven many researchers in various science and engineering disciplines. Early works from this field can be found in \cite{AastroemEykhoff71, NarendraParthasarathy90}. Two popular approaches in that regard either use symbolic regression to compose functions out of simple operations and functions~\cite{Bongard2007, Schmidt2009}, or apply sparse regression on a set of function candidates and combine them linearly~\cite{Brunton2016, Champion19, Chen2021}. In \cite{Conti2024}, coordinates and coefficients are identified in a generative manner to generate reduced-order models with embedded uncertainty quantification.
Moreover, Gaussian processes have been used for the identification of ordinary differential equations (ODEs) \cite{BeckersEtAl22}, differential-algebraic equations (DAEs) \cite{ZaspelGuenther24}, PDEs in~\cite{Raissi2017} and for ODEs describing the Euler--Lagrange equations in~\cite{Offen2024}.
Another frequent topic of interest are approaches for discovery under partial observations~\cite{Bakarji2023}.
Many model discovery approaches seek to identify effective coordinate systems, e.g., ones in which nonlinear dynamics appear linear~\cite{Lusch2018, Brunton2022}, see Koopman theory, or ones that are low-dimensional \cite{Champion19, Conti2023}.
For a comprehensive overview of parsimonious model discovery, we refer to \cite{Kutz2022, Brunton2023}.

For the model discovery of pH systems, there is active research, e.g, to derive models in the frequency domain~\cite{Schwerdtner2021} or in the time domain using an adapted version of dynamic mode decomposition~\cite{Morandin23} or to infer linear pH realizations using input--output data from the time domain in~\cite{Cherifi2022} or in~\cite{Ortega2023} for single-input, single-output systems. In \cite{GoyalEtAl23}, the authors use a pH-like decomposition of the system matrix to infer stable quadratic models through operator inference which was introduced in~\cite{PEHERSTORFER2016}. Moreover,~\cite{Desai21} introduced pH neural networks (PHNNs) that parameterize the flow of low-dimensional pH systems as neural networks and consequently do not obtain the underlying system operators. Similarly, \cite{Neary2023} uses neural networks to create low-dimensional pH systems while making use of the network modeling aspect of pH systems to split the problem into simple subsystems. On the other hand, in \cite{Salnikov2023}, machine learning is used to identify a connectivity structure to reformulate a known unstructured ODE into a pH system. In \cite{Yildiz2024}, the authors identify a quadratic \textit{Hamiltonian} system from data by using an autoencoder to either lift nonlinear data or reduce high-dimensional data while weakly enforcing a symplectic structure in the latent space.

\paragraph{Structure-preserving MOR for Port-Hamiltonian Systems}
Port-Hamiltonian systems have desirable system-theoretical properties and allow for modular network modeling. Structure-preserving MOR preserves these properties even in the reduced space, allowing the combination of full- and reduced-order models and enabling stable long-term simulations. Many general projection-based methods have been adapted to a structure-preserving variant for pH systems. There exist structure-preserving methods of tangential interpolation \cite{BeattieEtAl22,GugercinEtAl12}, spectral factorization \cite{BreitenUnger22}, moment-matching \cite{PolyugavanderSchaft10,IonescuAstolfi13,EggerEtAl18} and balanced truncation \cite{GuiverOpmeer13,BreitenEtAl22,BorjaEtAl23}. Power-conserving methodologies that reduce the differential geometric Dirac structure are described in \cite{PolyugavanderSchaft12, HauschildEtAl19}. Optimization-based approaches have been developed in \cite{MoserLohmann20, SchwerdtnerVoigt23}. A sensitivity analysis of structure-preserving projection methods for the application of a classical guitar is given in \cite{Rettberg23}.
In the special case of Hamiltonian systems, techniques from randomized linear algebra have been used in \cite{Herkert23} to reduce the computational cost of basis generation.

The model reduction of nonlinear pH systems is part of active research, and many open questions remain persistent. Nevertheless, early results can be found in \cite{BeattieGugercin11, ChaturantabutEtAl16}, where a POD-based and an $\mathcal{H}_2$ quasi-optimal selection of subspaces is developed. Generalizations of balanced truncation to nonlinear pH systems exist in \cite{KawanoScherpen18, SarkarScherpen23}. In \cite{Liljegren-Sailer20, LiljegrenSailer2022}, MOR is performed for a nonlinear flow problem and in \cite{Schulze23} for general transport-dominated phenomena. Recent work performs structure-preserving MOR through a nonlinear separable approximation ansatz \cite{Schulze2023}. Finally, an autoencoder-based reduction that requires knowledge of the FOM operators is described in \cite{LepriEtAl23}.

\paragraph*{Autoencoder-based (non-)intrusive MOR}
When MOR is used to derive low-dimensional surrogate models, one classification is whether the reduction uses explicit access to the underlying operators (intrusive MOR) or not (non-intrusive MOR).
Non-intrusive MOR is relevant if the underlying operators cannot be accessed, e.g., since the operators are contained in a closed-source software package, as is often the case in commercial software.

Autoencoder-based, \emph{intrusive} MOR has been discussed, e.g., in \cite{Kashima2016, Hartmann2017, Lee20}.
These works spawned various extensions, e.g., with a focus on hyperreduction via shallow masked autoencoders \cite{Kim2022},
knowledge-infused formulations for Hamiltonian systems \cite{Buchfink23,Brantner2023},
specific constrained autoencoders that define a projection \cite{Otto2023},
and differential geometric formulations \cite{FalS18,Otto2023,Buchfink2024}.

Autoencoder-based, \emph{non-intrusive} methods have been developed in, e.g.,~\cite{Fresca21, Fresca2022, Cote2023, Kneifl2023, kneifl2024, Pichi2024}. They solely rely on data to create black-box surrogate models capturing the essential dynamics of the data-generating model. Various autoencoder architectures are used in this area, such as convolutional~\cite{Fresca21}, graph-convolutional~\cite{Gruber2022, Pichi2024,kneifl2024}, or variational ones~\cite{Kneifl2023}. Still, also linear projections like the POD~\cite{Hesthaven2018} or combinations of POD with an autoencoder~\cite{Fresca2022} are deployed to create suitable low-dimensional latent spaces. A knowledge-infused approach for Hamiltonian systems has been discussed in \cite{Sharma2022} for classical operator inference, in \cite{Cote2023} based on autoencoders and Hamiltonian neural networks, and in \cite{Sharma2023} for a quadratic approximation.

\section{Method}
This section describes in detail the methodology developed for identifying pH systems. For this purpose,
we give a general formulation of system identification and non-intrusive MOR in our problem setup (\Cref{sec:problem_setup}),
introduce pH systems and the corresponding desirable system-theoretical properties (\Cref{sec:port_Hamiltonian}),
discuss our framework for identifying pH dynamics (\Cref{sec:identification_pH}) and embed it in a non-intrusive model reduction scheme.
Finally, we discuss how system-theoretical properties transfer through decoding to the state space (\Cref{sec:systheo_properties_preserved}).

\subsection{Problem Setup: System identification and non-intrusive MOR}
\label{sec:problem_setup}

We aim to identify a (possibly parametric) initial value problem from data,
which is relevant in system identification (with full-state observations) and non-intrusive MOR.
Additionally, we impose the structure of a pH system to guarantee desirable system-theoretical properties like passivity for the identified model.

In the following, we motivate and describe our problem setting.
All considered vector spaces are finite-dimensional.
Objects that we assume to be given are marked with (G$\ast$),
while the objects to be identified are marked with ($\stateSpace\ast$) and ($\latentStateSpace\ast$).
We assume to be given
\begin{itemize}
    \item[(G1)] an $\stateDim$-dimensional vector space $\stateSpace \cong \Rdim^{N}$, referred to as \emph{state space},
    \item[(G2)] a \emph{time interval} $\timeSpace \vcentcolon= [0, \timeFinal]$ with $\timeFinal > 0$,
    \item[(G3)] (optionally) a \emph{parameter set} $\paramSpace \subset \Rdim^{\paramDim}$, $\paramDim \in \N_0$, and \label{given:parameter set}
    \item[(G4)] $\numData \in \N$ \emph{(full-state) observations} $\{ \statesTraj[t_i; \params_i] \}_{i=1}^{\numData} \subset \stateSpace$ at different \emph{time steps} $\{ t_i \}_{i=1}^{\numData} \subset \timeSpace$ and \emph{parameter vectors} $\{ \params_i \}_{i=1}^{\numData} \subset \paramSpace$.
\end{itemize}

Then, our problem setting is to identify
\begin{itemize}
    \item[($\stateSpace$1)] a (possibly parametric) right-hand side $\rhsState: \stateSpace \times \timeSpace \times \paramSpace \to \stateSpace$, and
\end{itemize}
that defines the initial value problem
\begin{align}\label{eq:system_state}
    \begin{split}
        \statesApproxDotTraj &= \rhsState (\statesApproxTraj, t; \params )\\
        \statesApproxTraj[0; \params] &= \statesApprox_0(\params)= \statesTraj[0; \params]
    \end{split}
     & \text{such that for all $i \in \{1, \dots, \numData \}$:}              &
     & \statesApproxTraj[t_i; \params_i] \approx \statesTraj[t_i; \params_i],
\end{align}
i.e., the identified initial value problem should be able to describe the data given in (G4).
For high-dimensional spaces, $N \gg 1$, identifying the objects ($\stateSpace 1$) becomes computationally intensive or even intractable.
This is typically counteracted by additionally assuming to be given
\begin{itemize}
    \item[(G5)] a low-dimensional vector space $\latentStateSpace \cong \Rdim^{\latentStateDim}$, $\dim(\latentStateSpace) =\vcentcolon \latentStateDim \ll \stateDim$,
        referred to as the \emph{latent state space},
\end{itemize}
to identify
\begin{itemize}
    \item[($\latentStateSpace$1)] a (possibly parametric) right-hand side $\rhsLatentState: \latentStateSpace \times \timeSpace \times \paramSpace \to \latentStateSpace$,
    \item[($\latentStateSpace$2)] a (possibly parametric) initial value $\phState_0: \paramSpace \to \latentStateSpace$, and
    \item[($\latentStateSpace$3)] a mapping $\fullDecoder \in \contFun^1(\latentStateSpace, \stateSpace)$, referred to as the \emph{decoder},
        and a mapping $\fullEncoder \in \contFun^1(\stateSpace, \latentStateSpace)$,
        referred to as the \emph{encoder},
\end{itemize}
that define a low-dimensional initial value problem on the latent space with
\begin{align}\label{eq:low_dim_model}
    \begin{split}
        \phStateDotTraj &= \rhsLatentState (\phStateTraj, t; \params )\\
        \phState(0, \params) &= \underbrace{\fullEncoder(\states_0(\params))}_{=\vcentcolon \phState_0(\params)}
    \end{split}
     & \text{such that for all $i \in \{1, \dots, \numData \}$:}                         &
     & \fullDecoder( \phStateTraj[t_i; \params_i] ) \approx \statesTraj[t_i; \params_i].
\end{align}
This problem setting is subject to two different research areas depending on properties of the data-generating system and the data $\{ \statesTraj[t_i; \params_i] \}_{i=1}^{\numData} \subset \stateSpace$ itself:
(i) If the data-generating system is unknown or explicit access to the underlying operators is not possible, this task is referred to \emph{data-driven system identification};
(ii) If the data-generating system is high-dimensional but the operators are not accessible, e.g., due to closed-source software, the problem setting can be attributed to the field of \emph{non-intrusive MOR}.

To implement prior physical knowledge or mathematical guarantees, the right-hand side can be modified to incorporate the desired properties, which is referred to \emph{assuming additional structure (in the right-hand side)}.
A typical workflow in non-intrusive MOR is that this structure is already present in the data-generating system.
If the same structure is reflected in the low-dimensional initial value problem \eqref{eq:low_dim_model},
such techniques are referred to as \emph{structure-preserving MOR techniques}.
In practice, however, the structure may not be known for the data-generating system.
If the objects ($\latentStateSpace\ast$) exist with additional structure in the right-hand side $\rhsLatentState$,
we refer to our method as a \emph{structure-revealing MOR technique}. In the present paper, the structure is the one of pH systems.

The prototypical system relevant in the following is an initial value problem with inputs and outputs.
Thus, we additionally assume to be given
\begin{itemize}
    \item[(G6)] an $\phPortDim$-dimensional vector space $\inputSpace \cong \Rdim^{\phPortDim}$ for the inputs and the corresponding vector space $\phOutputSpace \cong \Rdim^{\phPortDim}$ for the outputs, which are in combination typically referred to an \emph{(input--output) port}, and
    \item[(G7)] an arbitrary \emph{input signal} $\phInputU: \timeSpace \to \inputSpace$.
\end{itemize}
Instead of identifying a general right-hand side $\rhsLatentState$ as in ($\latentStateSpace$1), the goal is to identify
\begin{itemize}
    \item[($\latentStateSpace$4)] an \emph{input function} $\funInput: \inputSpace \times \paramSpace \to \latentStateSpace$, and an \emph{output function} $\funOutput: \latentStateSpace \times \paramSpace \to \phOutputSpace$, and
    \item[($\latentStateSpace$5)] a time-independent function $\rhs: \latentStateSpace \times \paramSpace \to \latentStateSpace$ describing the \emph{state-dependent part of the right-hand side}
\end{itemize}
to formulate an initial value problem on the latent space with inputs and outputs
\begin{align}\label{eq:low_dim_sys_w_in_out}
    \begin{split}
        \phStateDotTraj &= \rhs(\phStateTraj; \params )
        + \funInput( \phInputUTraj ; \params )\\
        \phStateTraj[0, \params] &= \phState_0(\params),\\
        \phOutputTraj &= \funOutput(\phStateTraj; \params ),
    \end{split}
\end{align}
which is a special case of \eqref{eq:low_dim_model} with the respective choice of the right-hand side $\rhsLatentState$.
In the following, we mention the dependency on the parameter vector $\params$ only if it is relevant in the specific context and skip it otherwise for the sake of brevity.

\subsection{Port-Hamiltonian Systems}
\label{sec:port_Hamiltonian}
Port-Hamiltonian systems are energy-based models, which offer the advantage that energy can be used to interpret system dynamics and prove desirable system-theoretical properties.
In comparison to classical energy-based approaches like Lagrangian and Hamiltonian systems,
pH systems naturally allow to include (a)~dissipative terms and (b)~a port to interact with other models or controllers.
To reflect the two extensions in the balance of energy, pH systems formulate the energy balance with a dissipation inequality, which states that the system energy is bounded by the energy transferred through the ports.
The ports are especially useful (b1)~for network modelling to connect different physical domains physically consistently \cite{SchaftJeltsema14,Mehrmann23,Duindam09},
(b2)~hierarchical modeling since the clear definition of ports allows to replace parts of the model seamlessly,
or (b3)~for connecting to controllers.
Examples of the combination of pH systems with control are
control through interconnection~\cite{Schaft20},
interconnection and dissipation assignment passivity-based control (IDA-PBC)~\cite{OrtegaGarcia-Canseco04},
and measures for the distance to instability and its optimization~\cite{GillisEtAl18, GillisSharma17}.
Another advantage of pH systems is that the energy-based nature of pH systems allows to guarantee important system-theoretical properties like passivity, stability, and boundedness of solutions under mild assumptions (see \Cref{sec:mathematical_properties}).

The general framework of pH systems allows for various representations.
Among others, there are linear time-variant, descriptor, nonlinear and infinite-dimensional pH systems \cite{SchaftJeltsema14, Mehrmann23, Duindam09}.
In the present work, we identify linear time-invariant (LTI) pH systems (also referred to as linear input-output Hamiltonian system with dissipation, e.g., in \cite[Chapter~10]{SchaftJeltsema14}) on the latent space.
They are (possibly parametric) initial value problems with inputs and outputs \eqref{eq:low_dim_sys_w_in_out}, which are defined by
\begin{itemize}
    \item a symmetric and positive definite matrix $\phQ(\params) \in \Rdim^{\phDim \times \phDim}$ referred to as the \emph{energy matrix},
          \begin{align}
              \phQ(\params) = \phQ(\params)^\transpose \succ 0,  \label{eq:ph_property_Q}
          \end{align}
          which defines a scalar-valued function $\Hamiltonian: \latentStateSpace \times \paramSpace \to \Rdim,\, (\phStatePoint, \params) \mapsto \frac{1}{2}\phStatePoint^\transpose \phQ(\params) \phStatePoint$ referred to as the \emph{Hamiltonian} or \emph{energy},
    \item a skew-symmetric matrix $\phJ(\params) \in \Rdim^{\phDim \times \phDim}$ referred to as the \emph{structure matrix},
          \begin{align}
              \phJ(\params) = -\phJ(\params)^\transpose,  \label{eq:ph_property_J}
          \end{align}
    \item a symmetric and positive semidefinite matrix $\phR(\params)\in \Rdim^{\phDim \times \phDim}$ referred to as the \emph{dissipation matrix},
          \begin{align}
              \phR(\params) = \phR(\params)^\transpose \succeq 0  \label{eq:ph_property_R}
          \end{align}
    \item a matrix $\phB(\params)\in\Rdim^{\phDim\times\phPortDim}$ referred to as the \emph{port matrix},
\end{itemize}
formulating an initial value problem with inputs and outputs of the form
\begin{equation}
    \begin{aligned}
        \label{eq:linear_ph_system}
        \phStateDotTraj
        =                          & \big( \phJ(\params) - \phR(\params) \big) \phQ(\params) \phStateTraj
        + \phB(\params) \phInputUTraj,                                                                    \\
        \phStateTraj[0, \params] = & \phState_0(\params),                                                 \\
        \phOutputTraj
        =                          & \phB(\params)^\transpose \phQ(\params) \phStateTraj.
    \end{aligned}
\end{equation}
Skipping the dependency on the parameter vector $\params$ for the sake of brevity,
it can be shown that the change in the energy stored in the system
\begin{align}\label{eq:dissipation_inequality}
    \ddt\left[
        \Hamiltonian(\phStateTraj[t])
        \right]
    = \rT{\phOutputTraj[t]} \phInputUTraj[t]
    \underbrace{
        - \rT{(\phQ \phStateTraj[t])}  \phR (\phQ \phStateTraj[t])
    }_{
        \leq 0
    }
    \leq \rT{\phOutputTraj[t]} \phInputUTraj[t]
\end{align}
can be split in a dissipated part characterized by $\phR$ and the power of the port $\rT{\phOutputTraj[t]} \phInputUTraj[t]$.
Due to the positive semidefiniteness of $\phR$, this change can be bounded from above by the energy of the port,
referred to as a \emph{dissipation inequality}.
By integrating the above equation, we know that the energy present in the system $\Hamiltonian(\phStateTraj[t])$ is bounded by the energy present at the beginning $\Hamiltonian(\phStateTraj[0])$ plus the energy that entered or left through the port
\begin{align*}
    \Hamiltonian(\phStateTraj[t]) \leq \Hamiltonian(\phStateTraj[0]) + \int_0^{\timee}  \rT{\phOutputTraj[s]} \phInputUTraj[s] \,\mathrm{d}s.
\end{align*}

A more general formulation is the nonlinear pH system (also referred to as \emph{input--output Hamiltonian system with dissipation} \cite[Chapter~10]{SchaftJeltsema14}), where the Hamiltonian $\Hamiltonian$ can depend on the pH state non-quadratically
and the pH system matrices $\phJ$, $\phR$, and $\phB$ can depend on the state $\phJ\in \contFun(\phSpace;\Rdim^{\phDim \times \phDim})$, $\phR\in \contFun(\phSpace;\Rdim^{\phDim \times \phDim})$, and $\phB\in \contFun(\phSpace;\Rdim^{\phDim \times \phPortDim})$.
Skipping possible parameter dependency and denoting the state dependency with $\evalField[{{\phStateTraj[\timee]}}]{(\cdot)}$,
the nonlinear pH system reads
\begin{equation}
    \begin{aligned}
        \phStateDotTraj[\timee] & = (\evalField[{{\phStateTraj[\timee]}}]{\phJ} - \evalField[{{\phStateTraj[\timee]}}]{\phR}) \nabla\Hamiltonian(\phStateTraj[\timee]) + \evalField[{{\phStateTraj[\timee]}}]{\phB} \phInputUTraj[\timee], \quad \phStateTraj[0] = \phState_0 \in \phSpace \\
        \phOutputTraj[\timee]   & = \rT{\evalField[{{\phStateTraj[\timee]}}]{\phB}} \nabla\Hamiltonian(\phStateTraj[\timee]),
        \label{eq:nonlinear_ph_system}
    \end{aligned}
\end{equation}
where the properties of the system matrices from the listing above are satisfied point-wise
\begin{align}
    \evalField[\phStatePoint]{\phJ} & = -\evalField[\phStatePoint]{\phJ}^\transpose,         &
    \evalField[\phStatePoint]{\phR} & = \evalField[\phStatePoint]{\phR}^\transpose\succeq 0, &
                                    & \forall\phStatePoint\in\phSpace.
    \label{eq:nonlinear_ph_properties}
\end{align}
By construction, the dissipation inequality \eqref{eq:dissipation_inequality} holds also for the nonlinear pH system \eqref{eq:nonlinear_ph_system}.

\subsubsection{System-theoretical Properties of Port-Hamiltonian Systems}\label{sec:mathematical_properties}

The dissipation inequality \eqref{eq:dissipation_inequality} is the main ingredient for multiple desirable system-theoretical properties,
which we discuss in the following.

\paragraph*{Passivity}
The initial value problem with inputs and outputs \eqref{eq:low_dim_sys_w_in_out} is called \emph{passive}, if there exists a function $\funStorage: \latentStateSpace \to \Rdim$,
which (i) is non-negative, i.e. $\funStorage(\phStatePoint) \geq 0 \;\forall \phStatePoint \in \latentStateSpace$,
and (ii) satisfies the \emph{dissipation inequality}
$\ddt \left[ \funStorage(\phStateTraj[t]) \right] \leq \rT{\phInputUTraj[t]} \phOutputTraj[t]$
(see, e.g., \cite[Sec.~7]{SchaftJeltsema14}).
The function $\funStorage$ is typically referred to as a \emph{storage function} and the product $\rT{\phInputUTraj[t]} \phOutputTraj[t]$ is typically understood as the \emph{supplied power}.
For pH systems with a Hamiltonian bounded from below, i.e., there exists $\HamiltonianMin \in \Rdim$ such that $\Hamiltonian(\phStatePoint) \geq \HamiltonianMin$ for all $\phStatePoint \in \phSpace$, this is fulfilled naturally by choosing $\funStorage(\phStatePoint) = \Hamiltonian(\phStatePoint) - \HamiltonianMin$ due to the dissipation inequality \eqref{eq:dissipation_inequality}.

\paragraph*{Lyapunov Stability}
The storage function of a passive system can be used to obtain so-called Lyapunov stable points.
To formulate these, we assume that $\norm{\cdot}: \latentStateSpace \to \Rdim_{\geq 0}$ is a norm on $\latentStateSpace$.

\begin{mydef}[Equilibrium Point and Lyapunov Stability, e.g., {\cite[Sec.~12]{Meyer17}}]
    For an initial value problem \eqref{eq:low_dim_sys_w_in_out} with zero inputs, i.e., $\inputs(t) = 0$ for all $\timee \in \timeSpace$,
    a point $\stateEquilibrium \in \latentStateSpace$ is called an \emph{equilibrium point} if $\rhs(\stateEquilibrium) = 0$.
    An equilibrium point $\stateEquilibrium \in \latentStateSpace$ is called \emph{Lyapunov stable}
    if for all $\varepsilon > 0$ there exists $\delta > 0$ such that
    $\norm{\stateEquilibrium - \phState_0} \leq \delta$
    implies that
    $\norm{\stateEquilibrium - \phStateTraj[t; \phState_0]} \leq \varepsilon$ for all $t \in \timeSpace$,
    where $\phStateTraj[t; \phState_0]$ denotes the solution of \eqref{eq:low_dim_sys_w_in_out} with initial value $\phState_0$.
\end{mydef}

\begin{mythm}[Lyapunov's Stability Theorem, e.g., {\cite[Thm.~12.1.1]{Meyer17}}]\label{thm:lyapunov_stability}
    Consider a domain $\boundedNeighbourhood \subset \latentStateSpace$ with an equilibrium point $\stateEquilibrium \in \boundedNeighbourhood$.
    If there exists a function $V: \boundedNeighbourhood \to \Rdim$ such that
    \begin{enumerate}
        \item[(i)] $\stateEquilibrium$ is a strict local minimum, i.e., $V(\stateEquilibrium) < V(\phStatePoint)$ for all $\phStatePoint \in \boundedNeighbourhood \setminus \{ \stateEquilibrium \}$, and
        \item[(ii)] $\ddt \left[ V(\phStateTraj[t]) \right] \leq 0$ along solutions $\phStateTraj[t] \in \boundedNeighbourhood$,
    \end{enumerate}
    then $\stateEquilibrium$ is Lyapunov stable and $V$ is referred to as a \emph{Lyapunov function}.
\end{mythm}

An equilibrium point of a passive system with zero inputs, $\phInputUTraj[t] = 0$ for all $t \in \timeSpace$, which is additionally a strict local minimum of the storage function $\funStorage$, is automatically Lyapunov stable with $\funStorage$ as a Lyapunov function,
since the assumption (ii) in \Cref{thm:lyapunov_stability} is fulfilled by the dissipation inequality.

\paragraph*{Boundedness of Solutions}
Under additional assumptions, it can be shown that the solution of a pH system is contained in a bounded neighborhood.

\begin{mythm}[Boundedness of Solutions]\label{thm:boundedness_of_solutions}
    For a given pH system \eqref{eq:linear_ph_system} or \eqref{eq:nonlinear_ph_system} with initial value $\phState_0 \in \latentStateSpace$,
    assume there exists a bounded open neighborhood $\boundedNeighbourhood \subset \phSpace$ such that
    \begin{enumerate}
        \item[(i)] the initial value is within the neighborhood, i.e., $\phState_0 \in \boundedNeighbourhood$,
        \item[(ii)] the Hamiltonian on the boundary is higher than at the initial value, i.e., $\Hamiltonian(\phStatePoint) > \Hamiltonian(\phState_0)$ for all $\phStatePoint \in \partial \boundedNeighbourhood$,
        \item[(iii)] there is no input, $\phInputUTraj[t] = 0$ for all $t \in \timeSpace$.
    \end{enumerate}
    Then the solution is contained in this neighborhood for all times, $\phStateTraj[t] \in \boundedNeighbourhood$ for all $t \in \timeSpace$ and thus, there exists $\constBoundedSolutions \geq 0$ such that $\norm{\phStateTraj[t] - \phState_0}_{\latentStateSpace} \leq \constBoundedSolutions$ for all $\timee \in \timeSpace$.
\end{mythm}

A proof for this theorem can be found in \cref{App Proof Boundedness of Solutions}. Roughly speaking, the boundary of the neighborhood $\boundedNeighbourhood$ builds a barrier of higher energy than the initial value, which cannot be passed since the input into the system is zero.

\subsection{Identification of Port-Hamiltonian Dynamics}
\label{sec:identification_pH}

We introduce two methods to discover pH systems from data:
(i) The \emph{port-Hamiltonian identification network (\phin{})} in a nonparametric and a parametric setting,
and (ii) the \emph{Autoencoder-based port-Hamiltonian Identification Network (\aphin{})} which uses an autoencoder especially to tackle high-dimensional problems.
The \phin{} does not use an autoencoder, which is why for this case the variables $\statesTraj = \latentStateTraj$ and $\statesApproxTraj = \phStateTraj$ from the notation in \Cref{sec:problem_setup} can be identified by setting $\stateSpace = \latentStateSpace$ and using the identity $\text{id}_{\latentStateSpace}: \latentStateSpace \to \latentStateSpace, \latentStatePoint \mapsto \latentStatePoint$ as decoder and encoder, i.e., $\fullDecoder \equiv \fullEncoder \equiv \text{id}_{\latentStateSpace}$.
To match the notation of the \aphin{} in the later sections, we use the variable names $\latentStateTraj$ and $\phStateTraj$ for the coordinates in \phin{}.

To understand the underlying concept,
we briefly describe how a general linear system with input can be learned in a nonparametric setting without autoencoder.
Assume to be given full-state observations $\{ \latentStateTraj[t_i] \}_{i=1}^{\numData} \subset \latentStateSpace$, their time-derivatives $\left\{ \latentStateDotTraj[t_i] \right\}_{i=1}^{\numData} \subset \latentStateSpace$, and input signals $\{ \phInputUTraj[t_i] \}_{i=1}^{\numData} \subset \inputSpace$ at different {time steps} $\{ t_i \}_{i=1}^{\numData} \subset \timeSpace$.
To discover the system dynamics, the right-hand side $\rhsLatentState$ can be parameterized by network parameters $\weights$ also called \emph{weights} which are optimized to match the given data.
The weights are used to parameterize operators on the right-hand side, e.g., by assembling a dense $n \times m$ matrix $\M$ from weights $\weights_{\M}$ via
\begin{align*}
    \mathcal{M}_{n \times m}:
    \Rdim^{n \cdot m} \to \Rdim^{n \times m},\;                           &
    \weights_{\M} \mapsto
    \left[\begin{array}{ccc}
                  \weight_{\M, 1}                 & \cdots & \weight_{\M, m}         \\
                  \vdots                          &        & \vdots                  \\
                  \weight_{\M, (n-1) \cdot m + 1} & \cdots & \weight_{\M, n \cdot m}
              \end{array}\right], &
                                                                          & \weights_{\M} = [\weight_{\M, 1}, \dots, \weight_{\M, n \cdot m}].
\end{align*}
Then,
an initial value problem with inputs and outputs \eqref{eq:low_dim_sys_w_in_out} based on a linear state-dependent part of the right-hand side and linear input function can be parameterized with
\begin{align}
    \phStateDotTraj[\timee]
                                                                               & = \rhsLatentState(\phStateTraj[\timee], t)
    \vcentcolon= \ltiAid \phStateTraj[\timee] + \ltiBid \phInputUTraj[\timee], &
    \ltiAid                                                                    & \vcentcolon= \mathcal{M}_{\latentStateDim \times \latentStateDim}(\weights_{\A}), &
    \ltiBid                                                                    & \vcentcolon= \mathcal{M}_{\latentStateDim \times \phPortDim}(\weights_{\phB}),
    \label{eq: lin approx}
\end{align}
where the weights~$\weights_{\A} \in \Rdim^{\latentStateDim^2}$, $\weights_{\phB} \in \Rdim^{\latentStateDim \cdot \phPortDim}$, $\weights = \rTsb{ \rT{\weights_{\A}}, \; \rT{\weights_{\phB}} }\in\Rdim^{\latentStateDim \cdot (\latentStateDim+\phPortDim)}$ parameterize the system matrix and input matrix respectively.
To discover system dynamics that match the given data,
the loss function
\begin{align*}
    \loss_{\text{LTI}}(\weights)
    \vcentcolon=
    \sum_{i=1}^{\numData} \norm{\latentStateDotTraj[\timee_i] - \left(\ltiAid_{\weights} \latentStateTraj[\timee_i] + \ltiBid_{\weights} \phInputUTraj[\timee_i] \right)}^2_2
\end{align*}
is optimized, where we denote the dependency on the weights as a subindex $(\cdot)_{\weights}$ if we present the parameterized matrices in combination with the loss.
This setup aims to discover the unknown system matrices~$\ltiAid$ and $\ltiBid$ by optimizing the loss function using gradient-based optimization methods such as ADAM \cite{Kingma14}.
More generally, the right-hand side $\rhsLatentState$ might have a more complex dependence in the weights $\weights$, e.g., parameterized by an artificial neural network.
Thus, we refer to this approach as an \emph{identification network} and to \eqref{eq: lin approx} as an \emph{LTI identification network}.

In this work, however, we are not just interested in the identification of general LTI systems but especially in the identification of linear pH systems. Therefore, we construct an identification framework that automatically generates pH systems per the architecture's design.

\paragraph*{\phin{}: Port-Hamiltonian Identification Network}
The idea of moving from the LTI identification framework to a pH identification network is based on adding more structure to the identified system matrices.
Accordingly, the system matrix~$\ltiAid$ is replaced by three matrices~$\phJid,\, \phRid,\, \phQid$ to take the respective pH matrices into account.
Each of the according matrices is parameterized by weights $\weights_{\phJ}\in\Rdim^{\latentStateDim_{\mathrm{skew}}}$, $\weights_{\phR}\in\Rdim^{\latentStateDim_{\mathrm{sym}}}$, and $\weights_{\phQ}\in\Rdim^{\latentStateDim_{\mathrm{sym}}}$,
where $\latentStateDim_{\mathrm{skew}} = \frac{\latentStateDim(\latentStateDim-1)}{2}$ and $\latentStateDim_{\mathrm{sym}} = \frac{\latentStateDim(\latentStateDim+1)}{2}$ define the number of entries in a skew-symmetric or symmetric matrix of size $\latentStateDim \times \latentStateDim$.
Based on the two operators that assemble a lower triangular matrix\footnote{Abbreviated as “tril” in consistency with the corresponding MATLAB and NumPy commands.}
\begin{align}
     & \mathcal{M}_{\mathrm{tril}}: \, \mathbb{R}^{\latentStateDim_{\mathrm{sym}}} \to \mathbb{R}^{\latentStateDim \times \latentStateDim}, \; \weights \mapsto \left[\begin{array}{cccc}
                                                                                                                                                                              \weight_1 &           &        & \bm{0}                                   \\
                                                                                                                                                                              \weight_2 & \weight_3 &        &                                          \\
                                                                                                                                                                              \vdots    & \vdots    & \ddots &                                          \\
                                                                                                                                                                              \cdots    & \cdots    & \cdots & \weight_{\latentStateDim_{\mathrm{sym}}}
                                                                                                                                                                          \end{array}\right]
\end{align}
or a strict lower triangular matrix
\begin{align}
     & \mathcal{M}_{\mathrm{stril}}: \, \mathbb{R}^{\latentStateDim_{\mathrm{skew}}} \to \mathbb{R}^{\latentStateDim \times \latentStateDim}, \; \weights \mapsto \left[\begin{array}{cccc}
                                                                                                                                                                                0         &        &                                           & \bm{0} \\
                                                                                                                                                                                \weight_1 & 0      &                                           &        \\
                                                                                                                                                                                \vdots    & \ddots & \ddots                                    &        \\
                                                                                                                                                                                \cdots    & \cdots & \weight_{\latentStateDim_{\mathrm{skew}}} & 0
                                                                                                                                                                            \end{array}\right] \;,
\end{align}
the pH matrices are parameterized as
\begin{align}
    \Weights_{\phJ} & \vcentcolon= \mathcal{M}_{\mathrm{stril}}(\weights_{\phJ}),        &
    \phJid          & \vcentcolon= \Weights_{\phJ} - \rT{\Weights_{\phJ}}                &
                    & \implies                                                           &
    \phJid          & = -\phJid^\transpose,
    \\
    \Weights_{\phR} & \vcentcolon= \mathcal{M}_{\mathrm{tril}}(\weights_{\phR}),         &
    \phRid          & \vcentcolon= \Weights_{\phR} \rT{\Weights_{\phR}}                  &
                    & \implies                                                           &
    \phRid          & = \phRid^\transpose \succeq \mathbf{0},
    \\
    \Weights_{\phQ} & \vcentcolon= \mathcal{M}_{\mathrm{tril}}(\weights_{\phQ}),         &
    \phQid          & \vcentcolon= \Weights_{\phQ} \rT{\Weights_{\phQ}} + \regPosDef\eye &
                    & \implies                                                           &
    \phQid          & = \rT{\phQid} \succ \mathbf{0}.
\end{align}
By construction, the resulting matrix~$\phJid$ is skew-symmetric, and the matrices $\phRid$ and $\phQid$ are symmetric and positive (semi)definite.
To ensure positive definiteness for~$\phQid$, we add a regularization term $\regPosDef\eye$ with a regularization constant $\regPosDef=\numprint{1e-6}$ for single precision calculations.
A visualization of this approach is shown in \cref{fig: ph network}.

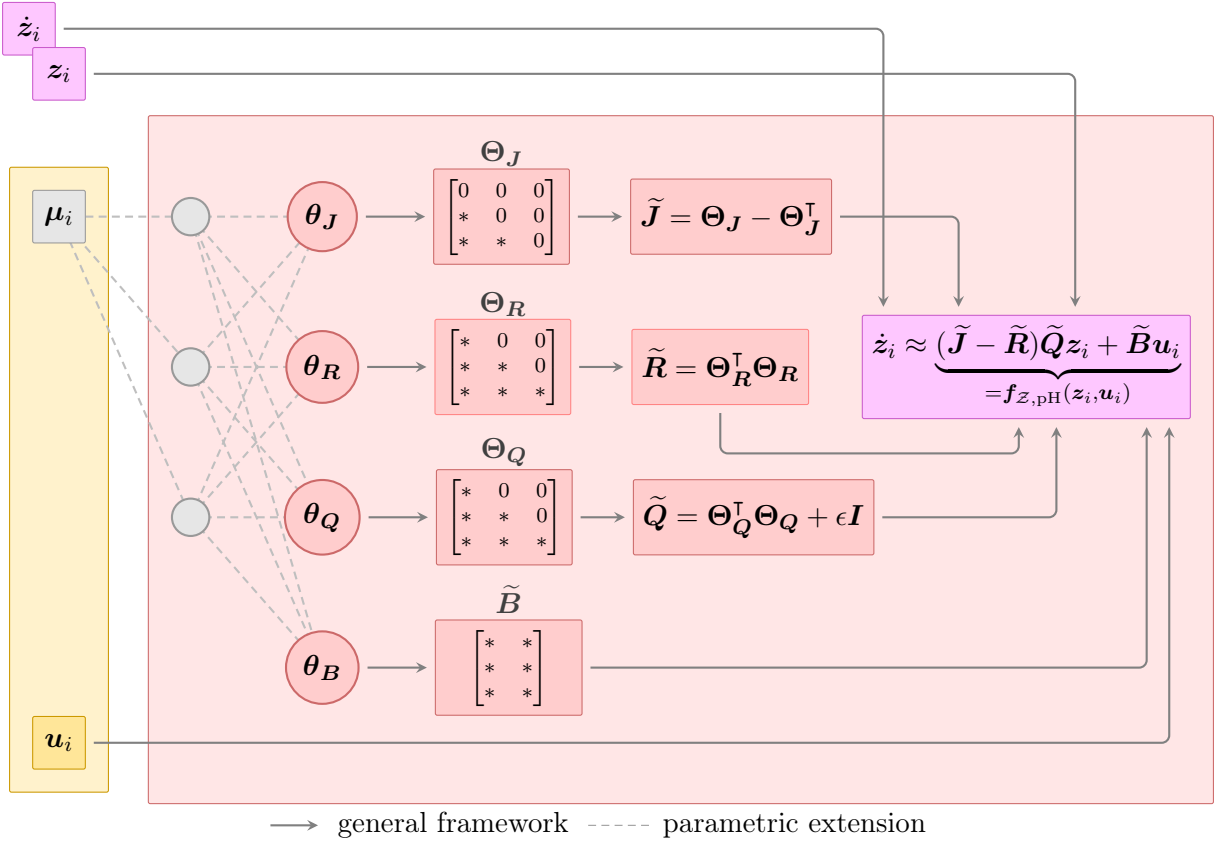
\begin{figure}[htb]
    \centering



\newcount\latentspace
\newcount\hideparams
\latentspace=1
\hideparams=0

\ifnum\latentspace=0
	\renewcommand{\latentState}{\states};
\fi
\begin{tikzpicture}[
  input_opacity/.style={opacity=1,},
  params_opacity/.style={opacity=1,},
  pca_opacity2/.style={opacity=\pcaOpacityRec,},
  ae_opacity/.style={opacity=\aeOpacity,},
]

\ifnum\hideparams=0
    \drawNodes{fcnn}{{{a}/gray, 
    		{,,}/lightgray, 
    		{,,,}/gray, 
    	}
    }
    
    \node [cell,
    params_opacity,
    draw=lightgray!80!black,
    fill=lightgray!40!white] (params) at (fcnn-1-1) {$\params_i$};
\fi
\ifnum\hideparams=1
    \drawNodes{fcnn}{{{$\mu$}/lightgray, 
    		{,,}/lightgray, 
    		{,,,}/lightgray, 
    	}
    }
\fi
    
\node [cell,
input_opacity,
draw=inputColor!80!black,
fill=inputColor!40!white] (inputs) at ($(params) - (0, 7)$) {$\inputs_i$};

\ifnum\hideparams=0
    \denselyConnectNodes{fcnn}{{1, 3, 4}}{lightgray}

    \node [cell,
        input_opacity,
        draw=latentColor!80!black, 
            fill=latentColor!40!white
        ] (statedot) at ($(params)+(-0.4, 2.5)$) {$\dot{\latentState}_i$};
        
    \node [cell,
        input_opacity,
        draw=latentColor!80!black, 
            fill=latentColor!40!white
        ] (state) at ($(params)+(0, 1.9)$) {$\latentState_i$};

    \node [neuron,
        input_opacity,
        draw=weightColor!80!black,
        fill=weightColor!40!white] (weightJ) at (fcnn-3-1) {$\weights_{\phJ}$};
    \node [neuron,
        input_opacity,
        draw=weightColor!80!black,
        fill=weightColor!40!white] (weightR) at (fcnn-3-2) {$\weights_{\phR}$};
    \node [neuron,
        input_opacity,
        draw=weightColor!80!black,
        fill=weightColor!40!white] (weightQ) at (fcnn-3-3) {$\weights_{\phQ}$};
    \node [neuron,
        input_opacity,
        draw=weightColor!80!black,
        fill=weightColor!40!white] (weightB) at (fcnn-3-4) {$\weights_{\phB}$};
\fi

\node[cell, right= of weightJ, 
    draw=weightColor!80!black,
    fill=weightColor!40!white, label={[label distance=-1mm, mylabel]: $\Weights_{\phJ}$},] 
        (offdiag1) {\scriptsize$\begin{bmatrix}0&0&0 \\ *&0&0 \\ *&*&0\end{bmatrix}$};
\node[cell, right=8mm of offdiag1, minimum height=10mm,
    draw=weightColor!80!black,
    fill=weightColor!40!white,] (J) {$\phJid=\Weights_{\phJ}-\Weights_{\phJ}^\transpose$};

\draw[Arrow] (weightJ) -- (offdiag1) ;
\draw[Arrow] (offdiag1) -- (J);

\node[cell, right=of weightR, 
    draw=weightColor,
    fill=weightColor!40!white, label={[label distance=-1mm, mylabel]: $\Weights_{\phR}$}] (diag2) {\scriptsize$\begin{bmatrix} * & 0 & 0 \\ * & * & 0 \\ * & * & * \end{bmatrix}$};
\node[cell, right=8mm of diag2, minimum height=10mm,
    draw=weightColor,
    fill=weightColor!40!white] (R) {$\phRid=\Weights_{\phR}^{\transpose} \Weights_{\phR}$};
\draw[Arrow] (weightR) -- (diag2);
\draw[Arrow] (diag2) -- (R);

\node[cell, right=of weightQ, 
    draw=weightColor!80!black,
    fill=weightColor!40!white, label={[label distance=-1mm, mylabel]: $\Weights_{\phQ}$}] (diag3) {\scriptsize$\begin{bmatrix} * & 0 & 0 \\ * & * & 0 \\ * & * & * \end{bmatrix}$};
\node[cell, right=8mm of diag3, minimum height=10mm,
    draw=weightColor!80!black,
    fill=weightColor!40!white] (Q) {$\phQid=\Weights_{\phQ}^{\transpose} \Weights_{\phQ} + \regPosDef\eye$};
\draw[Arrow] (weightQ) -- (diag3);
\draw[Arrow] (diag3) -- (Q);

\node[cell, minimum width=19.5mm, 
draw=weightColor!80!black,
fill=weightColor!40!white, right=of weightB, label={[label distance=-.5mm, mylabel]: ${\phBid}$}] (B) {\scriptsize$\begin{bmatrix}*&*\\\ast&*\\ \ast &*\end{bmatrix}$};
\draw[Arrow] (weightB) -- (B);

\node[cell, fill=latentColor!40!white, draw=latentColor!80!black, right=7mm of R] (result) {
    $\dot{\latentState}_i\approx
        \underbrace{
            ({\phJid} - {\phRid}) \phQid \latentState_i 
            + \phBid \inputs_i
        }_{= \rhsLatentStatePH(\latentState_i, \inputs_i)}
        $
    };
\draw[Arrow] (J) -| ($(result.north) - (0.9, 0)$);
\draw[Arrow] (R) |- +(1,-1.15) -| ($(result.south) - (0.1, 0)$);
\draw[Arrow] (Q) -| ($(result.south) + (0.4, 0)$);
\draw[Arrow] (state) -| ($(result.north) + (0.65, 0)$);
\draw[Arrow] (statedot) -| ($(result.north) - (1.9, 0)$);
\draw[Arrow] (inputs) -| ($(result.south) + (1.9, 0)$);
\draw[Arrow] (B) -| ($(result.south) + (1.6, 0)$);
\node[below=6mm of B] (belowB) {};
\node[above=3mm of weightJ] (aboveweightJ) {};

\begin{pgfonlayer}{bg1}
  \node[fit=(inputs) (params), cell, draw=inputColor!80!black, fill=inputColor!20!white, inner sep=0.3cm, label={[label distance=0mm, mylabel]:}] (box) {};
  \node[fit=(aboveweightJ) (fcnn-2-1) (result) (offdiag1) (belowB), cell, draw=phColor!80!black, fill=phColor!20!white, inner sep=0.3cm, label={[label distance=0mm, mylabel]:}] (box) {};
\end{pgfonlayer}

\end{tikzpicture}
\begin{tikzpicture}
\matrix[
    matrix of nodes,
    anchor=center,
    inner sep=0.1em,
    nodes={anchor=west},
    column sep=3pt,
    ] (legend) {
        \draw[Arrow] (0,0) -- (24pt,0); &  \node {general framework}; &&
        \draw[lightgray!80!black, densely dashed] (0,0) -- (24pt,0); &  \node {parametric extension}; &&
        \\
};
\end{tikzpicture}
    \caption{
        Schematic illustration of the port-Hamiltonian identification network (\phin{}).
        We denote dependency on the time step $\timee_i$ and parameter $\params_i$ as a subindex $(\cdot)_i$ for the sake of brevity.
    }
    \label{fig: ph network}
\end{figure}

Similar to the LTI case, we describe the right-hand side prescribing a pH system as
\begin{align}
    \phStateDotTraj[t] = {\rhsLatentStatePH}(\phStateTraj[t], \phInputUTraj[t])
    \vcentcolon= ( \phJid - \phRid ) \phQid \phStateTraj[t] + \phBid \phInputUTraj[t].
    \label{eq: ph network}
\end{align}
We identify the required matrices during training by optimizing the loss
\begin{equation}
    \loss_{\mathrm{pH}}(\weightsPH) = \sum_{i=1}^{\numData} \norm{\latentStateDotTraj[\timee_i]
        - \left(
        (\phJid_{\weightsPH} - \phRid_{\weightsPH}) \phQid_{\weightsPH} \latentStateTraj[\timee_i] + \phBid_{\weightsPH} \inputs(\timee_i)
        \right)}^2_2,
    \label{eq: ph loss}
\end{equation}
where we denote the dependency on the weights $\weightsPH \vcentcolon= \rTsb{ \rT{\weights_{\phJ}},\;\rT{\weights_{\phR}},\; \rT{\weights_{\phQ}},\; \rT{\weights_{\phB}}}$ for the pH matrices as subindex $(\cdot)_{\weightsPH}$ in this loss function.
We refer to this approach as \emph{port-Hamiltonian identification network~(\phin{})}.

However, the description~$\eqref{eq: ph network}$ is in general insufficient for parametric systems, e.g., if the model should respect changes in global parameters such as material properties, see Assumption (G3) in \cref{sec:problem_setup}.

\paragraph*{\phin{}, parameter-dependent case}%
Our \phin{} approach can be extended to learn a parameter-dependent pH system from parameter-dependent
full-state observations $\{ \latentStateTraj[t_i; \params_i] \}_{i=1}^{\numData} \subset \latentStateSpace$, their time-derivatives $\left\{ \latentStateDotTraj[t_i; \params_i] \right\}_{i=1}^{\numData} \subset \latentStateSpace$, and input signals $\{ \phInputUTraj[t_i] \}_{i=1}^{\numData} \subset \inputSpace$ at time steps $\{ t_i \}_{i=1}^{\numData} \subset \timeSpace$ and parameter vectors $\{ \params_i \}_{i=1}^{\numData} \subset \paramSpace$ that are considered constant throughout one trajectory. We account for these dependencies by allowing the pH matrices~$\phJid, \phRid, \phQid,$ and~$\phBid$ to depend on the parameter vector $\params$.
In detail, we use a deep fully-connected neural network $\neuralnetWeights$ parameterized by network weights $\weightsNeuralnetWeights$ that maps a parameter vector to weights for the pH matrices~$\weightsPH = \neuralnetWeights(\params)$.
The corresponding pH matrices are constructed from these parameter-dependent weights as explained in the previous section.
This results in a loss function
\begin{equation}
    \loss_{\mathrm{pH},\params}(\weightsNeuralnetWeights)
    =
    \sum_{i=1}^{\numData}
    \norm{\dot{\latentState}_i - \left(
        \left(
        \phJid_{\neuralnetWeights(\params_i; \weightsNeuralnetWeights)} - \phRid_{\neuralnetWeights(\params_i; \weightsNeuralnetWeights)}
        \right)
        \phQid_{\neuralnetWeights(\params_i; \weightsNeuralnetWeights)} \latentState_i
        + \phBid_{\neuralnetWeights(\params_i; \weightsNeuralnetWeights)} \inputs_i
        \right)}^2_2
\end{equation}
which optimizes the network weights $\weightsNeuralnetWeights$,
where we abbreviate $\dot{\latentState}_i \vcentcolon= \latentStateDotTraj[\timee_i; \params_i]$, ${\latentState}_i \vcentcolon= \latentStateTraj[\timee_i; \params_i]$, and $\inputs_i \vcentcolon= \phInputUTraj[\timee_i]$,
and denote the usage of the neural network $\weightsPH = \neuralnetWeights(\params)$ in the subindex of the according matrices together with the dependency on the network weights $\weightsNeuralnetWeights$.
In the following, we omit the dependency on the parameter vector~$\params$ for the sake of brevity.
This does not change the general procedure and can be added depending on the application.

\paragraph*{Dimensionality and Coordinate Limitations}
The restriction to approximate a system linearly may be too limiting in many cases and can lead to non-negligible deviations from the ground-truth data. However, following the Koopman theory, every nonlinear system can be expressed as a linear system in suitably high dimensional state space. For structured systems, this equivalence is not yet proven but is assumed to be reasonable, see e.g. \cite{Yildiz2024}.

At the same time, systems are often not given in a pH description, and even if they are, identifying a dynamical system for extremely high-dimensional systems leads to major challenges due to the vast possibilities of state interactions in high-dimensional system matrices. Hence, the \phin{} approach in its current formulation is limited for data possibly generated by (i) a non-pH system and (ii) by a highly nonlinear system. To circumvent these limitations, we embed the linear system identification into an autoencoder-based outer loop that allows to map the data from originally nonlinear systems to linear pH approximations through nonlinear transformations.

\begin{figure}[htb]
    \centering
    \input{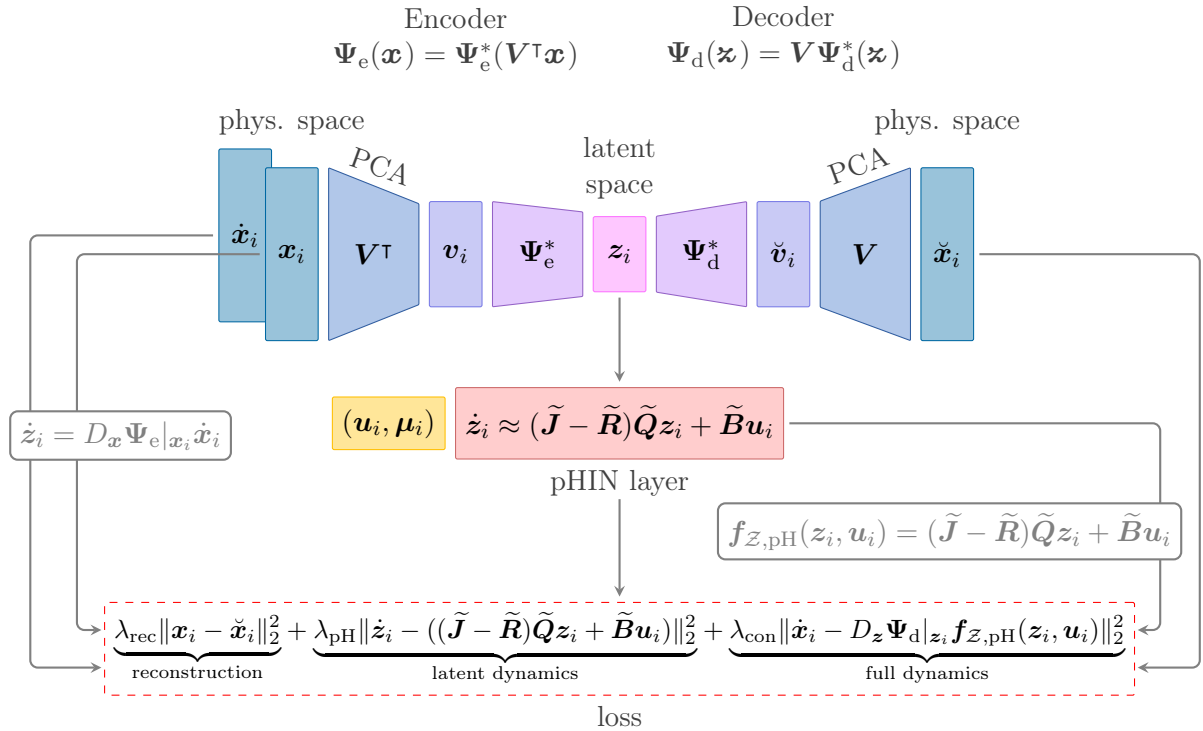}
    \caption{
        Schematic illustration of the \aphin{}~approach: An autoencoder composed of linear (PCA) and nonlinear (autoencoder) reduction identifies a low-dimensional manifold on which a linear pH system is identified on the latent space using our proposed \phin{} approach.
        We denote dependency on the time step $\timee_i$ and parameter $\params_i$ as a subindex $(\cdot)_i$ for the sake of brevity.}
    \label{fig:phae}
\end{figure}

\subsubsection{\aphin{}: An Autoencoder-based port-Hamiltonian Identification Network}
\label{sec:pH-autoencoder}
As motivated in \Cref{sec:intro},
the concept of coordinates and dynamics can be helpful in identifying a coordinate system in which the dynamics become simpler or low-dimensional via autoencoder-based model discovery.
In this section, we introduce the \emph{autoencoder-based port-Hamiltonian Identification Network (\aphin{})} to combine the idea of coordinates and dynamics with model discovery for pH systems.
Especially if the autoencoder performs a reduction in the dimension, this approach is attractive for non-intrusive model reduction.
In the following, we assume to be given data in the physical space $\stateSpace$ (which is not equal to the latent state space $\latentStateSpace$ as in the previous sections due to the autoencoder).
Thus, we assume to be given full-state observations $\{ \statesTraj[\timee_i] \}_{i=1}^{\numData}$,
velocities $\{ \statesDotTraj[\timee_i] \}_{i=1}^{\numData} \subset \stateSpace$, and
input signals $\{ \inputs(\timee_i) \}_{i=1}^{\numData} \subset \inputSpace$
at different time points $\{ \timee_i \}_{i=1}^{\numData} \subset \timeSpace$.
The data may also depend on parameter vectors $\params_i$, but we neglect this dependence in the following for the sake of brevity.

\paragraph*{Autoencoders: Coordinate Transformation and Dimension Reduction}
An autoencoder is a neural network that learns a low-dimensional representation of input data in the physical space $\stateSpace$ via a latent state space $\latentStateSpace$ based on two subnetworks: an encoder $\fullEncoder:\stateSpace\to\latentStateSpace$ and a decoder $\fullDecoder:\latentStateSpace\to\stateSpace$ both parameterized by network weights $\weightsAutoenc$.
For the input data $\statesTraj[\timee_i] \in \stateSpace$,
we denote the \emph{latent state} with $\latentStateTraj[\timee_i] \vcentcolon= \fullEncoder(\statesTraj[\timee_i])$
and the \emph{reconstructed latent state} with $\reconstruct{\states}(\timee_i) \vcentcolon= (\fullDecoder \circ \fullEncoder)(\statesTraj[\timee_i])$.
In our work, the autoencoder fulfills two main tasks:
(a)~For a low state space dimension $\stateDim \ll 1$, we choose $\latentStateSpace = \stateSpace$ and use the autoencoder to perform a nonlinear coordinate transformation.
(b)~For a high state space dimension $\stateDim \gg 1$, it additionally performs a dimension reduction by using a latent space $\latentStateSpace \cong \Rdim^{\latentStateDim}$ with $\latentStateDim \ll \stateDim$.
The autoencoder is trained by optimizing the reconstruction loss
\begin{align}
    \loss_{\text{rec}}(\weightsAutoenc)
    = \sum_{i=1}^{\numData} \big\lVert
    \statesTraj[\timee_i] -
    \underbrace{(\fullDecoderWithWeights \circ \fullEncoderWithWeights)(\statesTraj[\timee_i])}_{=\reconstruct{\states}_{\weightsAutoenc}(\timee_i)}
    \big\rVert^2_2 \;,
    \label{eq: rec loss}
\end{align}
where we denote the dependency on the network parameters $\weightsAutoenc$ in the subindex $(\cdot)_{\weightsAutoenc}$ in this loss.
While this loss already ensures that a suitable low-dimensional representation of a system state is found, it has no notion of pH dynamics. To acknowledge these attributes, a joint optimization of both reconstruction and pH loss is favorable.

\paragraph*{\aphin{}: Autoencoder-based Port-Hamiltonian Identification Network}
A low-dimensional system description enables an efficient identification of the corresponding pH dynamics that describe the system. Consequently, we embed \phin{} into the latent space to approximate the latent state of the autoencoder $\latentStateTraj[\timee] \in \latentStateSpace$ with a trajectory $\phStateTraj[\timee]  \in \latentStateSpace$ that follows pH dynamics (see \cref{fig:phae})
which defines an approximation $\statesApproxTraj[\timee] \vcentcolon= \fullDecoder(\phStateTraj[\timee])$ of the state $\statesTraj[\timee]$. By doing so, we introduce an inductive bias that ensures that the autoencoder prefers non-linear transformations according to which the latent dynamics appear both linear and in pH coordinates. This is loosely connected to the field of Koopman theory for dynamical systems~\cite{Lusch2018, Brunton2022}, which is built around the idea of finding a finite-dimensional coordinate system in which the dynamics of a nonlinear system appear approximately linear.

To simultaneously optimize the \phin{} and the autoencoder, we combine the pH loss~\cref{eq: ph loss} with the reconstruction loss~\cref{eq: rec loss} and add a consistency loss
\begin{align*}
    \loss_{\text{con}}(\weightsAutoenc, \weightsPH)
     & = \sum_{i=1}^{\numData} \Big\lVert
    \statesDotTraj[\timee_i]
    -
    \underbrace{
        \evalField[{{\latentStateTraj[\timee_i]}}]{\JacLatentState \fullDecoderWithWeights}
        \left(
        (\phJid_{\weightsPH}-\phRid_{\weightsPH})\phQid_{\weightsPH} \latentState(\timee_i)
        + \phBid_{\weightsPH} \inputs(\timee_i)
        \right)
    }_{\approx \statesApproxDotTraj[\timee_i]}
    \Big\rVert^2_2,
\end{align*}
which accounts for the dynamics in the original state space by comparing the reference time derivatives~$\dot{\states}$ with an approximation of $\dot{\approximate{\states}}$, as proposed in~\cite{Champion19}. This altogether results in the overall loss
\begin{equation}
    \begin{aligned}
        \loss(\weights) & = \lossfactor_{\text{rec}}\loss_{\text{rec}}(\weightsAutoenc) + \lossfactor_{\text{pH}}\loss_{\text{pH}}(\weightsPH) + \lossfactor_{\text{con}}\loss_{\text{con}}(\weights) \\
                        & =
        \sum_{i=1}^{\numData}
        \bigg(
        \lossfactor_{\text{rec}}\norm{\states(\timee_i)-\reconstruct{\states}_{\weightsAutoenc}(\timee_i)}^2_2                                                                                        \\
                        & \phantom{=
                \sum_{i=1}^{\numData}
            }
        + \lossfactor_{\text{pH}}\norm{\dot{\latentState}(\timee_i) - \left(
                (\phJid_{\weightsPH}-\phRid_{\weightsPH})\phQid_{\weightsPH} \latentState(\timee_i) + \phBid_{\weightsPH} \inputs(\timee_i)
        \right)}^2_2                                                                                                                                                                                  \\
                        & \phantom{=
                \sum_{i=1}^{\numData}
            }
        + \lossfactor_{\text{con}}
        \norm{
                \dot{\states}(\timee_i)
                -
                \evalField[{{\latentState(\timee_i)}}]{\JacLatentState \fullDecoderWithWeights}
                \left(
                (\phJid_{\weightsPH}-\phRid_{\weightsPH})\phQid_{\weightsPH} \latentState(\timee_i)
                + \phBid_{\weightsPH} \inputs(\timee_i)
                \right)
            }^2_2
        \bigg),
        \label{eq: loss}
    \end{aligned}
\end{equation}
where $\weights = \rTsb{\rT\weightsAutoenc, \rT\weightsPH}$ (or $\weights = \rTsb{\rT\weightsAutoenc, \rT\weightsNeuralnetWeights}$ for $\weightsPH = \neuralnetWeights(\params; \weightsNeuralnetWeights)$ in the parametric case) denotes the vector of all weights, the dependency on the network parameters is denoted as subindex $(\cdot)_{\weights}$ of the respective weights, and the loss factors $\lossfactor_{\text{rec}}$, $\lossfactor_{\text{pH}}$, $\lossfactor_{\text{con}} \geq 0$ are hyperparameters.
The loss can be supplemented by L1 regularization~$\loss_{\text{L1}}(\weightsPH)=\lossfactor_{\text{L1}}\norm{\weightsPH}_1$ on the \phin{} weights~$\weightsPH$ responsible for the pH matrices (or on the network weights $\weightsNeuralnetWeights$ in the parametric case with~$\loss_{\text{L1}}(\weightsNeuralnetWeights)=\lossfactor_{\text{L1}}\norm{\weightsNeuralnetWeights}_1$). In doing so, more sparse matrices can be enforced.

For a high state space dimension $\stateDim \gg 1$,
the training of autoencoder-based identification networks, in particular, the computation of~$\loss_{\text{con}}$, can be computationally expensive and memory demanding.
The computational cost can be reduced by using a linear reduction in the outermost layer of the encoder and the decoder
\begin{align}\
    \fullEncoder(\statesPoint)      & =\encoder(\pcaRed\statesPoint),     &
    \fullDecoder(\latentStatePoint) & =\pcaRec\decoder(\latentStatePoint)
\end{align}
where neural networks $\encoder$, $\decoder$ are used in combination with states $\intermediateState=\pcaRed\statesPoint, \reconstruct{\intermediateState} = \decoder(\latentStatePoint)\in\Rdim^{\intermediateStateDim}$ of intermediate dimension $\intermediateStateDim < \stateDim$ (see, e.g., \cite{Fresca2022}).
A common approach to obtain the projection matrix $\pcaRec\in\Rdim^{\stateDim\times\intermediateStateDim}$ for the linear reduction is principal component analysis (PCA, see, e.g., \cite{Volkwein2013}), which chooses the linear reduction to explain the maximum variance of the data.
By construction, however, this approach does not work for problems with slowly decaying Kolmogorov $n$-widths \cite{Buchfink2023b}. Consequently, the choice of the intermediate dimension is a trade-off between accuracy and computational cost.

For scenarios with a fast decaying Kolmogorov $n$-width, in turn, a purely linear reduction can be sufficient, making the nonlinear autoencoders obsolete. We refer to this approach as \aphin{} (linear), and both the nonlinear encoder and decoder simplify to the identity, i.e., $\encoder=\decoder=\text{id}$.

\paragraph*{Model Evaluation}\label{sec:model_evaluation}
Once the neural network weights are optimized, the discovered pH system in the latent space can be written in the form of a (parametric) LTI pH system as
\begin{subequations}
    \label{eq:pH-latent}
    \begin{align}
        \phStateDotTraj[t] & = (\phJid - \phRid) \phQid {\phStateTraj[t]} + \phBid \phInputUTraj[t] \;, \\
        \phOutput(t)       & = \phBid^\transpose \phQid {\phStateTraj[t]} \;,                           \\
        \phStateTraj[0]    & ={\phState}_0\coloneqq {\latentState}_0 = \fullEncoder(\states_0)
    \end{align}
\end{subequations}
with $\phJid = -\phJid^{\transpose} \in \Rdim^{\latentStateDim \times \latentStateDim}$, $\phRid = \phRid^{\transpose} \succeq \bm{0} \in \Rdim^{\latentStateDim \times \latentStateDim}$, $\phQid = \phQid^{\transpose} \succ \bm{0} \in \Rdim^{\latentStateDim \times \latentStateDim}$ and the encoded initial condition~$\approximate{\latentState}_0$. To preserve the pH structure and its properties also in discrete time, a structure-preserving time integration scheme, e.g. the implicit midpoint rule (IMR), is used \cite{KotyczkaLefevre19}.
Consequently, the transformation enables an efficient computation of the dynamics which shows significant accelerations compared to classical calculations in the high-dimensional state spaces.

This solution can then be transformed back into the physical space~$\statesApproxTraj[\timee]=\fullDecoder(\phStateTraj[\timee])$ via the decoder. In detail, this means that firstly, an initial condition is encoded, secondly, the latent solution for this initial condition and the corresponding inputs and parameters is calculated with a structure-preserving integrator and thirdly, the solution is decoded into the physical space. Please note that the temporal evolution of the latent states can be sensitive to the encoded initial condition. The workflow is visualized in \cref{fig:PH model}.

\begin{figure}[htb]
    \centering



\newcommand{\stateOpacity}{1}
\newcommand{\stateOpacityRec}{1}
\newcommand{\pcaOpacity}{1}
\newcommand{\pcaOpacityRec}{1}
\newcommand{\intermediateOpacity}{1}
\newcommand{\intermediateOpacityRec}{1}
\newcommand{\aeOpacity}{1}
\newcommand{\aeOpacityRec}{1}
\newcommand{\latentOpacity}{1}
\newcommand{\surrogateOpacity}{1}

\newcount\hidereconstruction
\newcount\hidereduction
\newcount\hideph
\newcount\aeOnly
\hidereconstruction=0
\hidereduction=0
\aeOnly=0
\hideph=0
\ifnum\hidereconstruction=1
	\renewcommand{\stateOpacityRec}{0.3}
    \renewcommand{\pcaOpacityRec}{0.3}
    \renewcommand{\intermediateOpacityRec}{0.3}
    \renewcommand{\aeOpacityRec}{0.3}
\fi
\ifnum\hidereduction=1
	\renewcommand{\stateOpacity}{0.3}
    \renewcommand{\pcaOpacity}{0.3}
    \renewcommand{\intermediateOpacity}{0.3}
    \renewcommand{\aeOpacity}{0.3}
\fi
\ifnum\hideph=1
    \renewcommand{\surrogateOpacity}{0.3}
\fi
\ifnum\aeOnly=1
	\renewcommand{\stateOpacity}{0.00}
    \renewcommand{\pcaOpacity}{0.0}
	\renewcommand{\stateOpacityRec}{0.0}
    \renewcommand{\pcaOpacityRec}{0.0}
    \renewcommand{\intermediateState}{\states}
\fi

\usetikzlibrary{backgrounds}

\begin{tikzpicture}[
    shorten >=1pt, shorten <=1pt,
    pca_opacity/.style={opacity=\pcaOpacity,},
    pca_opacity2/.style={opacity=\pcaOpacityRec,},
    ae_opacity/.style={opacity=\aeOpacity,},
    ae_opacity2/.style={opacity=\aeOpacityRec,},
    state_opacity/.style={opacity=\stateOpacity,},
    state_opacity2/.style={opacity=\stateOpacityRec,},
    intermediatestate_opacity/.style={opacity=\intermediateOpacity,},
    intermediatestate_opacity2/.style={opacity=\intermediateOpacityRec,},
    latent_opacity/.style={opacity=\latentOpacity,},
    surrogate_opacity/.style={opacity=\surrogateOpacity,},
  ]

\newcommand{\seperation}{.75mm}
    

    \node [cell, 
            draw=phColor!80!black,
            anchor=center,
            minimum height=9.5mm,
            surrogate_opacity, 
            fill=phColor!40!white,
            label={[label distance=0mm, surrogate_opacity, mylabel]below:pHIN layer}] 
            (ph) {$\dot{\approximate{\latentState}}(t)=(\phJid-\phRid)\phQid\phState(t) + \phBid \inputs(t)$};
    
    \node [cell, 
            label={[label distance=0mm, surrogate_opacity, mylabel]below:time integration},
            anchor=west,
            right=10mm of ph.east,
            minimum height=9.5mm, 
            ] (integration) {$\int$};

    \node [cell, 
            label={[label distance=0mm, surrogate_opacity, mylabel]},
            draw=latentColor,
            anchor=west,
            surrogate_opacity,
            right=10mm of integration.east,
            minimum height=9.5mm, 
            fill=latentColor!40!white] (phState) {$\approximate{\latentState}(t)$};

    \node [cell,
            draw=inputColor,
            anchor=east,
            surrogate_opacity,
            fill=inputColor!40!white,
            label={[label distance=0mm, surrogate_opacity, mylabel]}] (param) at ($(ph.west) + (-3*\seperation, 0.4)$) {$\params$};
            
    \node [cell,
        draw=inputColor,
        anchor=east,
        surrogate_opacity,
        fill=inputColor!40!white,
        label={[label distance=0mm, surrogate_opacity, mylabel]}] (inputs) at ($(ph.west) + (-3*\seperation, -0.4)$) {$\inputs$};


    \node [cell, 
            minimum height =23mm, 
            anchor=west,
            state_opacity,
            draw=stateColor,
            fill=stateColor!40!white] (state) at ($(ph.west) + (0, 2.5)$) {$\states_0$};
                    
    \node [trapez,
            pca_opacity,
            minimum width=23mm, 
            trapezium angle=65, 
            draw=intermediateLatentColor!80!black,  
            fill=intermediateLatentColor!40!white,
            anchor=east,
            right=\seperation of state.east,
            label={
            [label distance=0mm, pca_opacity, mylabel, rotate=-30]
            above:Encoder}
        ] (encoder) {$\fullEncoder$};
        
    \node [cell,
        draw=latentColor,
        pca_opacity,
        minimum height=9.5mm, 
        latent_opacity,
        anchor=east,
        right=\seperation of encoder.east,
        fill=latentColor!40!white] (red state) {$\latentState_0$};


    \node [trapez, 
        draw=pcaColor!80!black,  
        fill=pcaColor!40!white,
        ae_opacity2,
        trapezium angle=-65, 
            minimum width=23mm, 
            right=\seperation of phState.east,
            label={[label distance=0mm, ae_opacity2, mylabel, rotate=30]above: Decoder}
            ] (decoder) {$\fullDecoder$}; 
    
    \node [cell, label={[label distance=3mm, state_opacity2, mylabel]: phys. space}, 
            minimum height =23mm, 
            draw=stateColor,
            anchor=west,
            state_opacity2,
            text=black,
            right=\seperation of decoder.east,
            fill=stateColor!40!white] (stateReconstr) {${\approximate{\states}}(t)$};

\pgfdeclarelayer{bg1}    
\pgfsetlayers{bg1, main}
\begin{pgfonlayer}{bg1}
  \node[fit=(inputs) (param), cell, draw=inputColor!80!black, fill=inputColor!20!white, inner sep=0.1cm, label={[label distance=0mm, mylabel]:}] (box) {};
\end{pgfonlayer}

\draw [Arrow, shorten >= 1pt, surrogate_opacity, shorten <= 2pt] (ph.east) -- (integration.west);
\draw [Arrow, shorten >= 1pt, surrogate_opacity, shorten <= 2pt] (integration.east) -- (phState.west);
\draw [Arrow, shorten >= 2pt, surrogate_opacity, shorten <= 2pt] (red state.east) -| (integration.north);
\draw [Arrow, shorten >= 2pt, surrogate_opacity, shorten <= 2pt] (box.north) |- (state.west);

\end{tikzpicture}

    \caption{Model evaluation of the identified system: The identified (parametric) pH system is time-integrated using given inputs and the encoded initial conditions. Subsequently, the latent results are decoded into the high-dimensional space.
    }
    \label{fig:PH model}
\end{figure}
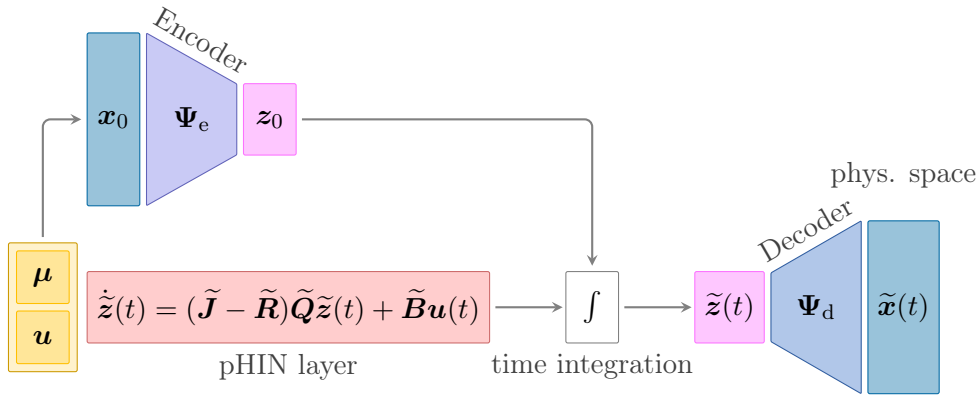

\subsection{Transfer of the Latent Model's System-theoretic Properties to the Physical Space}
\label{sec:systheo_properties_preserved}

So far, we have proposed a model discovery method to identify a pH model on the latent state space $\latentStateSpace$ from data on the state space $\stateSpace$ using an autoencoder.
Even if the structure-preserving system in latent space already has many advantages, it remains unclear whether the pH structure and the corresponding system-theoretic properties are also valid in the physical state space of the data. Consequently, we examine to what extent these properties are preserved in the following.

As detailed in \Cref{sec:mathematical_properties}, the dissipation inequality \eqref{eq:dissipation_inequality} accompanying our pH system on $\latentStateSpace$ can be used to show desirable system-theoretical properties like passivity, Lyapunov stability, and boundedness of solutions under additional assumptions.
However, by construction of our model, these properties hold on the latent space $\latentStateSpace$, while the observations are from the state space $\stateSpace$.
In this subsection, we discuss how the properties from \Cref{sec:mathematical_properties} transfer from $\phStateTraj[t] \in \latentStateSpace$ to $\statesApproxTraj[t] = \fullDecoder(\phStateTraj[t]) \in \stateSpace$.
This section is formulated for a general nonlinear pH system \eqref{eq:nonlinear_ph_system} with a Hamiltonian $\Hamiltonian: \latentStateSpace \to \Rdim$ and holds for the LTI pH system identified with \phin{} and a quadratic Hamiltonian $\Hamiltonian(\phStatePoint) = \tfrac{1}{2} \rT\phStatePoint \phQid \phStatePoint$.

For the upcoming analysis, we assume that the autoencoder fulfills
\begin{align}\label{eq:point_projection_property}
    (\fullEncoder \circ \fullDecoder)(\phStatePoint) & = \phStatePoint
    \quad \text{for all } \phStatePoint \in \latentStateSpace.
\end{align}
Differentiating this property directly yields by the chain rule that the multiplication of the Jacobians equals the identity, i.e.
\begin{align}\label{eq:tangent_projection_property}
    \evalField[\fullDecoder(\phStatePoint)]{\JacState \fullEncoder}
    \evalField[\phStatePoint]{\JacLatentState \fullDecoder}
     & = \eye \in \Rdim^{\latentStateDim \times \latentStateDim}
    \quad \text{for all } \phStatePoint \in \latentStateSpace.
\end{align}

\paragraph*{Port-Hamiltonian System on Physical Space}
We show that the pH system identified by \aphin{} on the latent state space $\latentStateSpace$ defines a nonlinear pH system on a subset of the state space $\stateSpaceApprox \vcentcolon= \fullDecoder(\latentStateSpace) \subset \stateSpace$.
Technically, $\stateSpaceApprox$ is a submanifold of $\stateSpace$ since $\stateSpaceApprox$ is homeomorphic to $\latentStateSpace \cong \Rdim^{\latentStateDim}$ by \eqref{eq:point_projection_property}.

\begin{mythm}
    Assume to be given a nonlinear pH system \eqref{eq:nonlinear_ph_system} with a solution trajectory $\phStateTraj[t] \in \latentStateSpace$ and output $\phOutputTraj[t] \in \phOutputSpace$.
    Then, the reconstructed solution trajectory $\statesApproxTraj[t] = \fullDecoder(\phStateTraj[t]) \in \stateSpaceApprox \subset \stateSpace$ is a solution to the nonlinear pH system
    \begin{equation}\label{eq:non_linear_ph_system_state_space}
        \begin{aligned}
            \statesApproxDotTraj[t]
                                 & = \left(
            \evalField[{{\statesApproxTraj[t]}}]{\phJApprox} - \evalField[{{\statesApproxTraj[t]}}]{\phRApprox}
            \right)
            \evalField[{{\statesApproxTraj[t]}}]{\nabla\HamiltonianApprox}
            + \evalField[{{\statesApproxTraj[t]}}]{\phBApprox} \phInputUTraj[t],                                                                          \\
            \statesApproxTraj[0] & = \phStateApprox_0 \in \stateSpace,                                                                                    \\
            \phOutputApprox(t)   & = \rT{\evalField[{{\statesApproxTraj[t]}}]{\phBApprox}} \evalField[{{\statesApproxTraj[t]}}]{\nabla\HamiltonianApprox}
        \end{aligned}
    \end{equation}
    with matrices
    \begin{align}\label{eq:ph_matrices_state_space}
                                                                      & \evalField[\statesApproxPoint]{\phJApprox}
        \vcentcolon= \evalField[\phStatePoint]{\JacLatentState \fullDecoder}
        \evalField[\phStatePoint]{\phJ}
        \rT{\evalField[\phStatePoint]{\JacLatentState \fullDecoder}}, &
                                                                      & \evalField[\statesApproxPoint]{\phRApprox}
        \vcentcolon= \evalField[\phStatePoint]{\JacLatentState \fullDecoder}
        \evalField[\phStatePoint]{\phR}
        \rT{\evalField[\phStatePoint]{\JacLatentState \fullDecoder}}, &
                                                                      & \evalField[\statesApproxPoint]{\phBApprox}
        \vcentcolon= \evalField[\phStatePoint]{\JacLatentState \fullDecoder}
        \evalField[\phStatePoint]{\phB},
    \end{align}
    for $\statesApproxPoint = \fullDecoder(\phStatePoint)$ with $\phStatePoint \in \latentStateSpace$,
    the reconstructed Hamiltonian  $\HamiltonianApprox \vcentcolon= \Hamiltonian \circ \fullEncoder$, the reconstructed initial value $\statesApprox_0 \vcentcolon= \fullDecoder(\phState_0)$, and the output $\phOutputApprox(t) = \phOutputTraj[t]$.
\end{mythm}
\begin{proof}
    The initial value is matched by construction.
    The time derivative of the reconstructed solution trajectory reads
    \begin{align*}
        \statesApproxDotTraj[t]
        = \ddt \fullDecoder(\statesApproxTraj[t])
        = \evalField[{{\phStateTraj[t]}}]{\JacLatentState\fullDecoder} \phStateDotTraj[t].
    \end{align*}
    Using that $\phStateTraj[t]$ solves the pH system \eqref{eq:nonlinear_ph_system} and the transposed identity for the Jacobians \eqref{eq:tangent_projection_property} yields
    \begin{align*}
        \statesApproxDotTraj[t]
         & = \evalField[{{\phStateTraj[t]}}]{\JacLatentState\fullDecoder}
        \Big(
        \left(
        \evalField[{{\phStateTraj[t]}}]{\phJ} - \evalField[{{\phStateTraj[t]}}]{\phR}
        \right)
        \underbrace{
            \rT{\evalField[{{\phStateTraj[t]}}]{\JacLatentState\fullDecoder}}
            \rT{\evalField[{{\fullDecoder(\phStateTraj[t])}}]{\JacState\fullEncoder}}
        }_{=\eye \in \Rdim^{\latentStateDim \times \latentStateDim}}
        \evalField[{{\phStateTraj[t]}}]{\nabla\Hamiltonian}
        + \evalField[{{\phStateTraj[t]}}]{\phB} \phInputUTraj[t],
        \Big),
    \end{align*}
    which is equivalent to \eqref{eq:non_linear_ph_system_state_space} by using that with the chain rule and \eqref{eq:point_projection_property}, it holds
    \begin{align*}
        \evalField[{{\statesApproxTraj[t]}}]{\nabla \HamiltonianApprox}
        = \evalField[{{\fullDecoder(\phStateTraj[t])}}]{\nabla (\Hamiltonian \circ \fullEncoder)}
        = \rT{\evalField[{{\fullDecoder(\phStateTraj[t])}}]{\JacState\fullEncoder}}
        \evalField[{{(\fullEncoder \circ \fullDecoder)(\phStateTraj[t])}}]{\nabla\Hamiltonian}
        = \rT{\evalField[{{\fullDecoder(\phStateTraj[t])}}]{\JacState\fullEncoder}}
        \evalField[{{\phStateTraj[t]}}]{\nabla\Hamiltonian}
    \end{align*}
    together with the definition of the system matrices \eqref{eq:ph_matrices_state_space}.
    Similarly, it holds for the output
    \begin{align*}
        \phOutputApprox(t) & = \rT{\evalField[{{\statesApproxTraj[t]}}]{\phBApprox}} \evalField[{{\statesApproxTraj[t]}}]{\nabla\HamiltonianApprox}
        = \rT{\evalField[{{\phStateTraj[t]}}]{\phB}}
        \underbrace{
            \rT{\evalField[{{\phStateTraj[t]}}]{\JacLatentState\fullDecoder}}
            \rT{\evalField[{{\fullDecoder(\phStateTraj[t])}}]{\JacState\fullEncoder}}
        }_{=\eye \in \Rdim^{\latentStateDim \times \latentStateDim}}
        \evalField[{{\phStateTraj[t]}}]{\nabla\Hamiltonian}
        = \phOutputTraj[t].
    \end{align*}
    By construction, $\evalField[\phStateApprox]{\phJApprox}$ is a skew-symmetric matrix, and $\evalField[\phStateApprox]{\phRApprox}$ is positive semi-definite.
    Thus, \eqref{eq:non_linear_ph_system_state_space} is indeed a nonlinear pH system on $\stateSpaceApprox$.
\end{proof}
Note that in the case that the matrices $\evalField[\phStatePoint]{\phJ}, \evalField[\phStatePoint]{\phR}, \evalField[\phStatePoint]{\phB}$ on the latent space are state-independent,
the matrices on the state space \eqref{eq:ph_matrices_state_space} are in general state-dependent due to the dependency on the Jacobians.
Thus, a linear pH system on the latent space can model a nonlinear pH system on $\stateSpaceApprox$.

\paragraph*{Dissipation Inequality on State Space}
It can be shown that the dissipation inequality also holds on the state space $\stateSpaceApprox$ by using
(a) that $\phOutputApprox(t) = \phOutputTraj[t]$,
(b) that the dissipation inequality holds for the system on the latent space $\latentStateSpace$,
and (c) the identity \eqref{eq:point_projection_property}:
\begin{align*}
    \rT{\phOutputApprox(t)} \phInputUTraj[t]
    \stackrel{\text{(a)}}{=} \rT{\phOutputTraj[t]} \phInputUTraj[t]
    \stackrel{\text{(b)}}{\leq}  \ddt \left[ \Hamiltonian( \phStateTraj[t] ) \right]
    \stackrel{\text{(c)}}{=}  \ddt \left[ (\Hamiltonian \circ \fullEncoder \circ \fullDecoder)( \phStateTraj[t] ) \right]
    = \ddt \left[ \HamiltonianApprox( \statesApproxTraj[t] ) \right].
\end{align*}

\paragraph*{Passivity on State Space}
If the Hamiltonian $\Hamiltonian$ of the latent system is bounded from below, i.e., $\Hamiltonian(\phStatePoint) \geq \HamiltonianMin$ for all $\phStatePoint \in \latentStateSpace$, it defines a storage function for which the system is passive (see \Cref{sec:mathematical_properties}).
In this case, the reconstructed Hamiltonian $\HamiltonianApprox = \Hamiltonian \circ \fullEncoder$ is also bounded from below for $\statesApproxPoint = \fullDecoder(\phStatePoint) \in \stateSpaceApprox$ due to \eqref{eq:point_projection_property} with
$$\HamiltonianApprox(\statesApproxPoint)
    = (\Hamiltonian \circ \fullEncoder \circ \fullDecoder) (\phStatePoint)
    = \Hamiltonian (\phStatePoint)
    \geq \HamiltonianMin.$$
Together with the dissipation inequality on $\stateSpaceApprox$ from the previous paragraph,
we then know that the pH system on the state space is passive with the storage function
$\funStorageApprox(\statesApproxPoint) = \HamiltonianApprox(\statesApproxPoint) - \HamiltonianMin \geq 0$.

\paragraph*{Lyapunov Stability on State Space}
If the Hamiltonian $\Hamiltonian$ of the latent system
(i) has a strict local minimum at $\stateEquilibrium \in \latentStateSpace$,
i.e., there exists a neighborhood $\domainLyapunov \subset \stateSpaceApprox$ of $\stateEquilibrium$ with
$\Hamiltonian(\phStatePoint) > \Hamiltonian(\stateEquilibrium)$
for all $\phStatePoint \in \domainLyapunov \setminus \{ \stateEquilibrium \}$,
and (ii) the input is zero, i.e., $\phInputUTraj[t] = 0$ for all $\timee \in \timeSpace$,
then we know that $\stateEquilibrium$ is a Lyapunov stable point for the system on the latent space with Lyapunov function $\Hamiltonian$
(see \Cref{sec:mathematical_properties}).
In this case, $\stateEquilibriumApprox \vcentcolon= \fullDecoder(\stateEquilibrium) \in \stateSpaceApprox$ is a strict local minimum of $\HamiltonianApprox$ since
it holds for the set $\domainLyapunovApprox \vcentcolon= \fullDecoder(\domainLyapunov)$
that for each $\statesApproxPoint \in \domainLyapunovApprox \setminus \{ \stateEquilibriumApprox \}$ there exists a $\phStatePoint \in \domainLyapunov \setminus \{ \stateEquilibrium \}$ with $\statesApproxPoint = \fullDecoder(\phStatePoint)$ and thus with \eqref{eq:point_projection_property}, it holds for all $\statesApproxPoint \in \domainLyapunovApprox \setminus \{ \stateEquilibriumApprox \}$ that
$$
    \HamiltonianApprox(\statesApproxPoint)
    = (\Hamiltonian \circ \fullEncoder \circ \fullDecoder)(\phStatePoint)
    = \Hamiltonian (\phStatePoint)
    < \Hamiltonian (\stateEquilibrium)
    = (\Hamiltonian \circ \fullEncoder \circ \fullDecoder)(\stateEquilibrium)
    = \HamiltonianApprox (\stateEquilibriumApprox).
$$
This proves that $\stateEquilibriumApprox$ is a Lyapunov stable point with Lyapunov function $\HamiltonianApprox$ for the system on $\stateSpaceApprox$ since
(i) $\stateEquilibriumApprox$ is a strict local minimum of $\HamiltonianApprox$,
and (ii) the dissipation inequality holds for $\HamiltonianApprox$ on $\stateSpaceApprox$.

\paragraph*{Boundedness of Solutions on State Space}
Consider an initial value $\phState_0$
and a bounded neighborhood $\boundedNeighbourhood \subset \latentStateSpace$
such that the assumptions (i)--(iii) of \Cref{thm:boundedness_of_solutions} are fulfilled.
Then, the solution on the latent space is bounded, i.e., there exists $C_{\latentStateSpace} \geq 0$
such that $\norm{ \phStateTraj[t] - \phState_0 }_{\latentStateSpace} \leq C_{\latentStateSpace}$ (see \Cref{sec:mathematical_properties}).
If $\fullDecoder$ is continuously differentiable on the closure $\overline{\boundedNeighbourhood}$,
we know that it is Lipschitz continuous with Lipschitz constant
$\lipschitzConst
    = \max_{\phStatePoint \in \overline{\boundedNeighbourhood}}
    \norm{ \evalField[\phStatePoint]{\JacLatentState \fullDecoder} }$,
i.e.,
$
    \norm{ \fullDecoder( \phStateTraj[t] ) - \fullDecoder( \phState_0 ) }_{\stateSpace}
    \leq
    \lipschitzConst
    \norm{ \phStateTraj[t] - \phState_0 }_{\latentStateSpace}$.
Then, we know that the solutions are bounded on the state space with
\begin{align*}
    \norm{ \statesApproxTraj[t] - \phStateApprox_0 }_{\stateSpace}
    =\norm{ \fullDecoder( \phStateTraj[t] ) - \fullDecoder( \phState_0 ) }_{\stateSpace}
    \leq \lipschitzConst \norm{ \phStateTraj[t] - \phState_0 }_{\latentStateSpace}
    \leq \lipschitzConst\, C_{\latentStateSpace}.
\end{align*}

\paragraph*{Note on the Assumptions that pH Structure is fulfilled in the Physical Domain}
The assumptions \eqref{eq:point_projection_property} and \eqref{eq:tangent_projection_property} are in general not fulfilled exactly.
In \cite{Otto2023}, a method is discussed to enforce these properties by considering constrained autoencoders by restricting the architectures of the autoencoder.
In contrast to that, it is argued in \cite[Thm.~6.4]{Buchfink2024} that the property \eqref{eq:point_projection_property} is approximately fulfilled for the generic setup of an autoencoder without special assumptions on the architecture if the autoencoder is Lipschitz continuous and the reconstruction loss \eqref{eq: rec loss} is low enough.
In the subsequent numerical examples, we follow the second approach and do not enforce the projection properties \eqref{eq:point_projection_property} and \eqref{eq:tangent_projection_property}. Nevertheless, we provide a numerical study of how well the projection properties are fulfilled.
\newcommand{\setTest}{\latentStateSpace_{\mathrm{test}}}

\section{Numerical Experiments}
In the following, we showcase the capabilities of our proposed \phin{} and \aphin{} approaches on several numerical examples. We start with a typical mechanical pH benchmark system in the form of a mass-spring-damper chain to validate the performance of \phin{} alone. We then test \aphin{} on a nonlinear low-dimensional system and a nonlinear high-dimensional input-affine and parametric system.
\begin{tcolorbox}[
        colback = black!1!white,
        colframe = black!10!white,
        left=2pt,
        right=2pt,
        top=2pt,
        bottom=2pt,
    ]
    The code used to generate the subsequent results is accessible via
    \begin{center}
        \href{https://github.com/Institute-Eng-and-Comp-Mechanics-UStgt/ApHIN}{https://github.com/Institute-Eng-and-Comp-Mechanics-UStgt/ApHIN}
    \end{center}
    with a persistent release in~\cite{darus-4446_2024} under MIT Common License, while the data can be found in~\cite{darus-4418_2024}.
\end{tcolorbox}

\paragraph*{Error Measures}
In order to validate the performance of our proposed methods, we use a relative error on the state space as
\begin{align}
    \errorState(t,\params) = \frac{\norm{\statesTraj[t,\params] -  \fullDecoder(\phStateTraj[t,\params])}_{2}}{\frac{1}{\nSamples}\sum_{i=1}^{\nSamples}\norm{\statesTraj[\timee_i,\params]}_{2}} \label{eq:error_state}
\end{align}
as well as one on the latent states
\begin{align}
    \errorLatent(t,\params) = \frac{\norm{\fullEncoder(\states(t,\params)) - \phStateTraj[t,\params]}_{2}}{\frac{1}{\nSamples}\sum_{i=1}^{\nSamples}\norm{\fullEncoder(\states(t_i,\params))}_{2}}
    \label{eq:error_latent}
\end{align}
normalized by the mean of the states and encoded states respectively over all time steps $\{ t_i \}_{i=1}^{\nSamples} \subset \timeSpace$ of a single simulation.

In addition, we consider the mean over all time steps $\{ t_i \}_{i=1}^{\nSamples} \subset \timeSpace$ and samples $\{ \params_j \}_{j=1}^{\nSims} \subset \paramSpace$ for the parameter vectors in the test dataset for both error quantities with
\begin{align}
    \errorStateMean = \frac{1}{\nSims\nSamples} \sum_{j=1}^{\nSims}\sum_{i=1}^{\nSamples} \errorState(\timee_i,\params_j) \label{eq:error_state_mean}
\end{align}
and
\begin{align}
    \errorLatentMean = \frac{1}{\nSims\nSamples} \sum_{j=1}^{\nSims}\sum_{i=1}^{\nSamples} \errorLatent(\timee_i,\params_j) \label{eq:error_latent_mean}
\end{align}
to quantify the results with scalar values.
Please note that we use a slight change of notation here with~$\timee_i$ and~$\params_j$ being pairwise distinct samples and the total amount of samples corresponding to $\numData=\nSims\nSamples$, whereas before the parameter and time instants were independent of each other, not restricted to such a Cartesian product structure.

Moreover, we want to evaluate to what extent the assumptions \cref{eq:point_projection_property} and \cref{eq:tangent_projection_property} necessary to transfer the pH properties from the latent to the state space are met. Accordingly, we check on the projection property with
\begin{align}
    \errorProj=
    \frac{1}{\nSims \nSamples}
    \sum_{j=1}^{\nSims}
    \sum_{i=1}^{\nSamples}
    \frac
    {\norm{
            \latentStateTraj[\timee_i; \params_j]
            - (\fullEncoder \circ \fullDecoder)(\latentStateTraj[\timee_i; \params_j])
        }_{2}^2}
    {\norm{
            \latentStateTraj[\timee_i; \params_j]
        }_{2}^2} \label{eq:projection_error}
    ,
\end{align}
and
\begin{align}
    \errorJac=
    \frac{1}{\nSims\nSamples}
    \sum_{j=1}^{\nSims}
    \sum_{i=1}^{\nSamples}
    \norm{
        \eye
        - \evalField[{{\fullDecoder(\latentStateTraj[\timee_i; \params_j])}}]{\JacState \fullEncoder}
        \evalField[{{\latentStateTraj[\timee_i; \params_j]}}]{\JacLatentState\fullDecoder}}_{2}^2. \label{eq:jacobian_error}
\end{align}

\subsection{Numerical example I: Mass-spring-damper chain}
The first numerical example serves as a benchmark model for the capability of the \phin{} method alone, without including dimensionality reduction or nonlinear coordinate transformations. This allows us to directly compare the identified matrices with corresponding reference matrices.
A one-dimensional mass-spring-damper chain connected by links, as visualized in \cref{fig:mass-spring-damper}, is chosen for this purpose as it is a typical pH benchmark system\footnote{A collection of pH benchmark systems can be found at \url{https://algopaul.github.io/PortHamiltonianBenchmarkSystems.jl/}}. It can be formulated as a pH system \cref{eq:linear_ph_system} using the displacements~$\disps$ and momenta~$\momentums$ as states. For a system with three links and input only for the leftmost mass, this results in~$\phStatePoint=[\disp_1, \disp_2, \disp_3, \momentum_1, \momentum_2, \momentum_3]$ where the according system matrices read

\begin{align*}
     & \hat{\phJ}=
    \begin{bmatrix}
        0  & 0  & 0  & 1 & 0 & 0 \\
        0  & 0  & 0  & 0 & 1 & 0 \\
        0  & 0  & 0  & 0 & 0 & 1 \\
        -1 & 0  & 0  & 0 & 0 & 0 \\
        0  & -1 & 0  & 0 & 0 & 0 \\
        0  & 0  & -1 & 0 & 0 & 0 \\
    \end{bmatrix},
     & \hat{\phR}=
    \begin{bmatrix}
        0 & 0 & 0 & 0          & 0          & 0          \\
        0 & 0 & 0 & 0          & 0          & 0          \\
        0 & 0 & 0 & 0          & 0          & 0          \\
        0 & 0 & 0 & \damping_1 & 0          & 0          \\
        0 & 0 & 0 & 0          & \damping_2 & 0          \\
        0 & 0 & 0 & 0          & 0          & \damping_3 \\
    \end{bmatrix}, \\
     & \hat{\phQ}=
    \begin{bmatrix}
        \stiffness_1  & -\stiffness_1             & 0                         & 0                 & 0                 & 0                 \\
        -\stiffness_1 & \stiffness_1+\stiffness_2 & -\stiffness_2             & 0                 & 0                 & 0                 \\
        0             & -\stiffness_2             & \stiffness_2+\stiffness_3 & 0                 & 0                 & 0                 \\
        0             & 0                         & 0                         & \frac{1}{\mass_1} & 0                 & 0                 \\
        0             & 0                         & 0                         & 0                 & \frac{1}{\mass_2} & 0                 \\
        0             & 0                         & 0                         & 0                 & 0                 & \frac{1}{\mass_3} \\
    \end{bmatrix}, \;
     & \hat{\phB}=
    \begin{bmatrix}
        0 \\
        0 \\
        0 \\
        1 \\
        0 \\
        0 \\
    \end{bmatrix}.
\end{align*}

\paragraph*{Experiments}
In a first step, we generate parametric training data resulting from varying parameter vectors $\params = (\mass, \stiffness, \damping)$ composed of the mass~$\mass_1= \mass_2= \mass_3=\mass\in[0.1, 100]$\,kg, stiffness~$\stiffness_1= \stiffness_2= \stiffness_3=\stiffness\in[0.1, 100]\,\frac{\text{N}}{\text{m}}$, as well as damping parameters~$\damping_1, \damping_2, \damping_3=\damping\in[0.1, 10]\,\frac{\text{Ns}}{\text{m}}$. The system is excited with a damped harmonic input signal~$\inputU(\timee)=e^{(-\delta\cdot\timee)}\sin(\omega\timee^2)$, where the frequencies~$\omega\in[0.5, 5]$\,Hz and decay rates~$\delta\in[0.125, 2]$ are chosen randomly.
Based on this data, we want to identify the system matrices using the approach explained in the previous section.
A pH description, however, is not unique.
Hence, we transform the ground-truth system to one with~$\phQ=\eye$ via Cholesky factorization of~$\hat{\phQ}$, see~\cite{Beattie2019}.
Consequently, we only identify the matrices~$\phJid, \phRid$ and~$\phBid$ with \phin{}. This means that even though the classic \phin{} approach works successfully, we fix~$\phQid=\eye$ to achieve comparability with the reference system matrices.
We then train \phin{} and fit it with respect to the loss \cref{eq: ph loss} so that the resulting weight matrices form a pH system and explain the data. The implementation details of the network and details on the dataset can be found in \cref{tab:network hyperparameter}.

\paragraph*{Results}
As can be seen in \cref{fig:msd-matrices}, the identified matrices match their reference quite well for various parameter samples, and the sparsity patterns of the matrices are met accurately. It can be seen that $\phJid$ is skew-symmetric and that $\phRid$ also has the positive-semidefinite properties of the original counterpart due to the positive entries on the diagonal. The identified model is able to determine that the input is only applied to the first mass by identifying zeros for all remaining entries of $\phBid$. The relative mean error in the (latent) state \eqref{eq:error_latent_mean} is $0.08$ in the training and $0.07$ in the test case, see~\cref{tab:perf}.
The training and evaluation time of the identified model can be found in \cref{tab: times}.

\begin{figure}[htb]
    \begin{subfigure}[c]{\linewidth}
        \centering



    \begin{tikzpicture}[
        ]
        \tikzstyle{spring}=[thick,decorate,decoration={zigzag,pre length=0.3cm,post
        length=0.3cm,segment length=6}]
       
        \tikzstyle{damper}=[thick,decoration={markings,  
          mark connection node=dmp,
          mark=at position 0.5 with 
          {
            \node (dmp) [thick,inner sep=0pt,transform shape,rotate=-90,minimum
        width=15pt,minimum height=3pt,draw=none] {};
            \draw [thick] ($(dmp.north east)+(2pt,0)$) -- (dmp.south east) -- (dmp.south
        west) -- ($(dmp.north west)+(2pt,0)$);
            \draw [thick] ($(dmp.north)+(0,-5pt)$) -- ($(dmp.north)+(0,5pt)$);
          }
        }, decorate]
       
        \tikzstyle{ground}=[fill,pattern=north east lines,draw=none,minimum width=0.3cm,minimum height=1.5cm,inner sep=0pt,outer sep=0pt]
       
        \node[draw,outer sep=0pt,thick] (M1) [minimum width=1.25cm, minimum height=1.25cm] {$\mass_1$};
        \node[draw,outer sep=0pt,thick, right = 2.5cm of M1.east] (M2) [minimum width=1.25cm, minimum height=1.25cm] {$\mass_2$};
        \draw[spring] ($(M1.east) - (0,0.5)$) -- ($(M2.west) - (0,0.5)$) 
        node [midway,above = 0.1cm] {$\stiffness_1$};
        \node (groundd1) [ground,
                            right = 1.5cm of $(M1.east)$,
                         minimum height=0.6cm,
                         yshift = 0.5cm] {};                         
        \draw[thick] ($(groundd1.north west)$) -- 
                         ($(groundd1.south west)$) {};
        \draw[damper] ($(M1.east) + (0,0.5)$) -- ($(groundd1.west)$) node [midway,above=0.3cm] {$\damping_1$};
        \draw[thick, dashed] ($(M1.north)$) -- ($(M1.north) + (0,1)$);
        \draw[ultra thick, -latex] ($(M1.north) + (-1,0.6)$) -- 
                                   ($(M1.north) + (0,0.6)$)
                                   node [midway, below] {$\inputU_1$};
        \draw[thick, dashed] ($(M1.south)$) -- ($(M1.south) + (0,-1)$);
        \draw[ultra thick, -latex] ($(M1.south) + (0,-0.5)$) -- 
                                   ($(M1.south) + (1,-0.5)$)
                                   node [midway, below] {$\disp_1$};
        \draw[thick, dashed] ($(M2.south)$) -- ($(M2.south) + (0,-1)$);
        \draw[ultra thick, -latex] ($(M2.south) + (0,-0.5)$) -- 
                                    ($(M2.south) + (1,-0.5)$)
                                    node [midway, below] {$\disp_2$};    
        \node (ground2) [ground,
                            right = 1.5cm of $(M2.east)$,
                        minimum height=0.6cm,
                        yshift = 0.5cm] {};
        \draw[thick] ($(ground2.north west)$) -- 
                        ($(ground2.south west)$) {};
        \draw[damper] ($(M2.east) + (0,0.5)$) -- ($(ground2.west)$) node [midway,above=0.3cm] {$\damping_2$};
        \node[draw,outer sep=0pt,thick,right = 2.5cm of M2.east] (Mnm) [minimum width=1.25cm, minimum height=1.25cm] {$\mass_{3}$};
        \draw[spring] ($(M2.east) - (0,0.5)$) -- ($(Mnm.west) + (0,-0.5)$) 
        node [midway,above = 0.1cm] {$\stiffness_2$};
        \node (groundend) [ground,
                            right = 2cm of $(Mnm.east)$,
                        minimum height=1.5cm,
                        ] {};
        \draw[thick] ($(groundend.north west)$) -- 
                        ($(groundend.south west)$) {};
        \draw[damper] ($(Mnm.east) + (0,0.5)$) -- ($(groundend.west)+ (0,0.5)$) node [midway,above=0.3cm] {$\damping_{3}$};
        \draw[spring] ($(Mnm.east) - (0,0.5)$) -- ($(groundend.west) - (0,0.5)$) 
            node [midway,above = 0.1cm] {$\stiffness_{3}$};
        \draw[thick, dashed] ($(Mnm.south)$) -- ($(Mnm.south) + (0,-1)$);
        \draw[ultra thick, -latex] ($(Mnm.south) + (0,-0.5)$) -- 
                                   ($(Mnm.south) + (1,-0.5)$)
                                   node [midway, below] {$\disp_{3}$};    
    \end{tikzpicture}

        \caption{A one-dimensional mass-spring-damper chain with 3 masses.}
        \label{fig:mass-spring-damper}
    \end{subfigure}
    \begin{subfigure}[c]{\linewidth}
        \centering
        \input{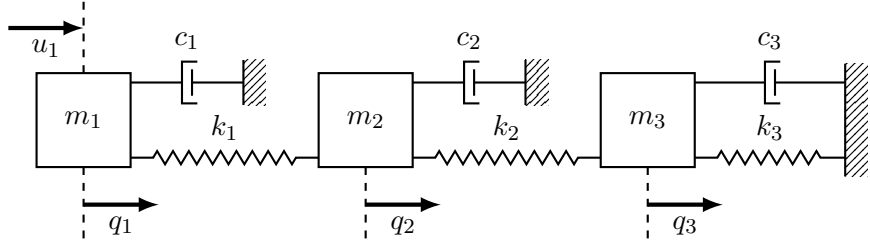}
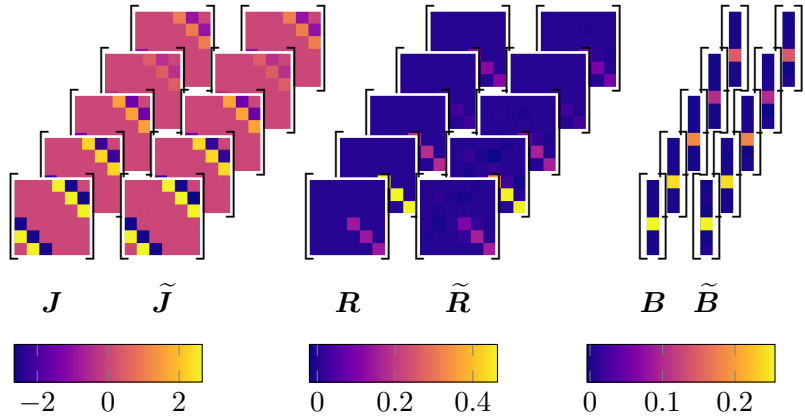
        \caption{Reference and identified port-Hamiltonian matrices for a three-link mass-spring-damper chain for five example parameter vectors~$\params$.}
        \label{fig:msd-matrices}
    \end{subfigure}
    \caption{A mass-spring-damper chain~\cref{fig:mass-spring-damper} that serves as a simple mechanical example for the port-Hamiltonian identification framework. The corresponding identified system matrices can be seen in~\cref{fig:msd-matrices}.}
\end{figure}

\subsection{Numerical example II: Pendulum}

With the second example, we want to showcase the benefits and necessity of incorporating the identification of a linear pH system into a nonlinear coordinate transformation. Therefore, we consider a mathematical pendulum as a comprehensive nonlinear mechanical example for which it is possible to derive the dynamics using the Hamiltonian formalism~\cite{Castanos2013}. A sketch of the pendulum is given in \cref{fig:pendulum}, where it can be seen that it consists of a single massless rod of length $l$ with a point mass attached on its far end. The Cartesian coordinates of the end point of the rod are given by~$(\pendulumPos_x,\; \pendulumPos_y)$. The small angle approximation is not assumed as the maximum deflection of the pendulum is chosen to be~$\vert \pendulumAngle \vert \leq \tfrac{\pi}{3}$.
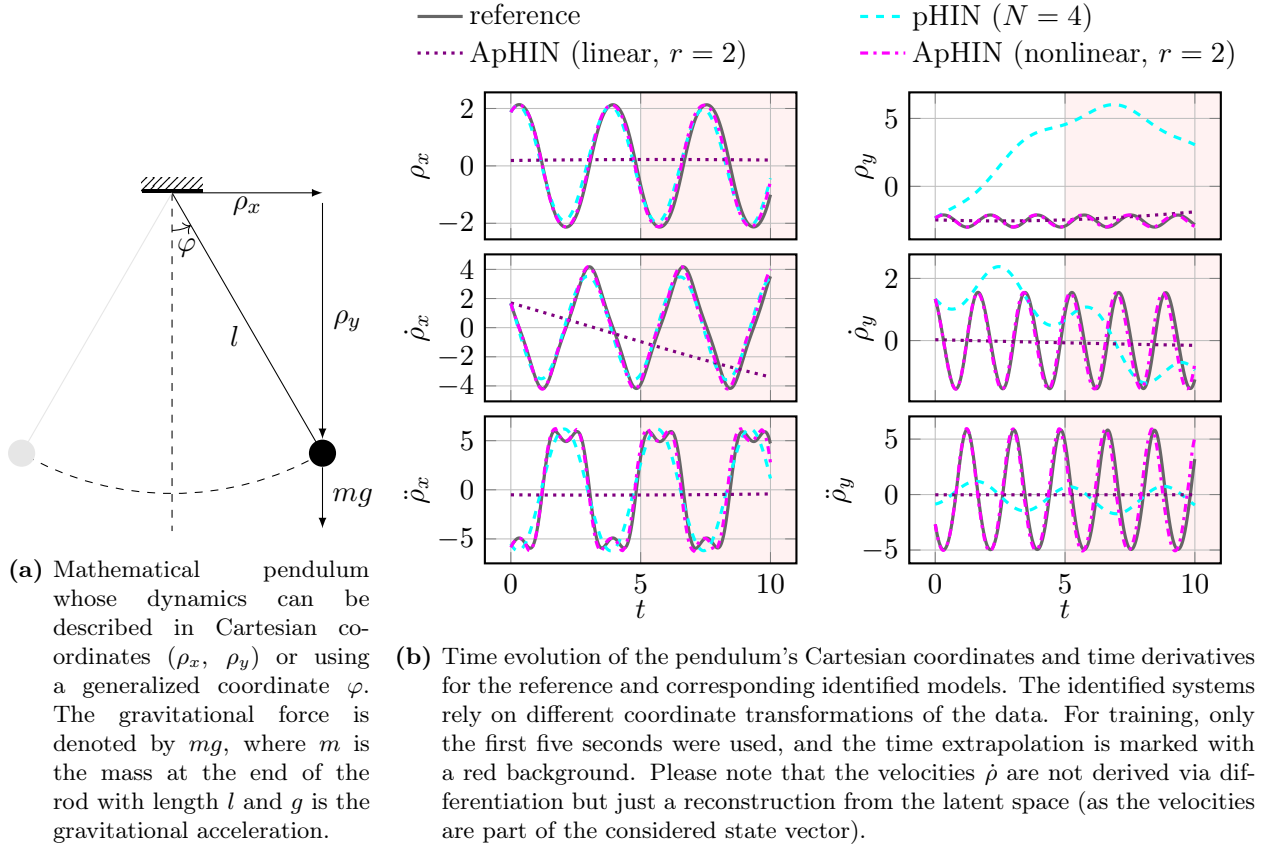
\begin{figure}[htb]
    \centering
    \begin{subfigure}[b]{0.29\textwidth}
        \begin{tikzpicture}
            \coordinate (centro) at (0,0);
            \fill (-.4,0) rectangle (.4,0.05);
            \fill[pattern=north east lines] (-.4,0) rectangle (.4,0.2);
            \draw[dashed] (-60:4) arc(-60:-120:4);
            \draw (0,0) -- node [xshift=-1ex, yshift=-1ex] (length) {$\pendulumLength$} (-60:4) node[fill,circle](mass){};
            \draw[-latex] (mass) -- node[right]{$\mass\gravitation$}++(0,-1) ;
            \draw[dashed,-] (centro) -- ++(0,-4.5) node (mid) {};
            \pic [draw, ->, "$\pendulumAngle$", angle eccentricity=1.5] {angle = mid--centro--mass};
            \draw[-latex] (centro) -- node [yshift=-1ex] (px) {$\pendulumPos_x$} ++(2,0) node (px) {};
            \draw[-latex] (px) -- node [xshift=2ex] (py) {$\pendulumPos_y$} (mass.north) node (py) {};
            \draw[black!10] (0,0) -- (-120:4) node[fill,circle]{};
        \end{tikzpicture}
        \caption{Mathematical pendulum whose dynamics can be described in Cartesian coordinates ($\pendulumPos_x$, $\pendulumPos_y$) or using a generalized coordinate $\pendulumAngle$. The gravitational force is denoted by $\mass\gravitation$, where $\mass$ is the mass at the end of the rod with length $\pendulumLength$ and $\gravitation$ is the gravitational acceleration. }
        \label{fig:pendulum}
    \end{subfigure}
    \hfill
    \begin{subfigure}[b]{0.69\textwidth}

\graphicspath{{fig/pendulum/}}


\begin{tikzpicture}

\newcommand{\linear}{linear.csv}
\newcommand{\pca}{pca.csv}
\newcommand{\nonlinear}{nonlinear.csv}
\newcommand{\reference}{reference.csv}
\newcommand{\trajectory}{4}

\pgfplotsset{
  myline/.style={
    very thick,
    dashdotted,
  },
  myaxis/.style={
    thick,
    xlabel style={yshift=1.ex},
    ylabel style={yshift=-1.ex},
    width=0.5\textwidth,
    height=0.5*0.618\textwidth,
    grid=major,
    axis background/.style=
    {%
        preaction={path picture={\fill [red!5] (axis cs:5,-10) rectangle (rel axis cs:1,10);}}
    },
  }
}
    \node (ax) at (0,0) {};

    \begin{axis}[
        myaxis,
        name=ax,
        ylabel={$\pendulumPos_x$},
        at={($ (ax.south) + (0,-10mm) $)},
        anchor=north,
        legend columns=2,
        legend cell align={left}, 
        xticklabels={},
        legend style={
            draw=none,
            at={(-0.25,1.1)},
            anchor=south west, 
            align=left,
            /tikz/column 2/.style={
                column sep=40pt,
            },
            }
    ]
        \addplot [myline, color=plotColor1, solid] table [col sep=comma, x=t, y=x_0_\trajectory] {fig/pendulum/\reference};
        \addplot [myline, color=plotColor4, dashed] table [col sep=comma, x=t, y=x_0_\trajectory] {fig/pendulum/\linear};
        \addplot [myline, color=plotColor3, dotted] table [col sep=comma, x=t, y=x_0_\trajectory] {fig/pendulum/\pca};
        \addplot [myline, color=plotColor2, dashdotted] table [col sep=comma, x=t, y=x_0_\trajectory] {fig/pendulum/\nonlinear};
        \addlegendentry{reference}
        \addlegendentry{\phin{} ($\stateDim=4$)}
        \addlegendentry{\aphin{} (linear, $\latentStateDim=2$)}
        \addlegendentry{\aphin{} (nonlinear, $\latentStateDim=2$)}
    \end{axis}
    \begin{axis}[
        myaxis,
        name=ax0,
        ylabel={$\pendulumPos_y$},
        at={($ (ax.east) + (15mm,0) $)},
        anchor=west,
        xticklabels={},
    ]
        \addplot [myline, color=plotColor1, solid] table [col sep=comma, x=t, y=x_1_\trajectory] {fig/pendulum/\reference};
        \addplot [myline, color=plotColor4, dashed] table [col sep=comma, x=t, y=x_1_\trajectory] {fig/pendulum/\linear};
        \addplot [myline, color=plotColor3, dotted] table [col sep=comma, x=t, y=x_1_\trajectory] {fig/pendulum/\pca};
        \addplot [myline, color=plotColor2, dashdotted] table [col sep=comma, x=t, y=x_1_\trajectory] {fig/pendulum/\nonlinear};
    \end{axis}

        \begin{axis}[
        myaxis,
        name=ax,
        ylabel={$\dot{\pendulumPos}_x$},
        at={($ (ax.south) + (0,-2mm) $)},
        anchor=north,
        xticklabels={},
    ]
        \addplot [myline, color=plotColor1, solid] table [col sep=comma, x=t, y=x_0_dt_\trajectory] {fig/pendulum/\reference};
        \addplot [myline, color=plotColor4, dashed] table [col sep=comma, x=t, y=x_0_dt_\trajectory] {fig/pendulum/\linear};
        \addplot [myline, color=plotColor3, dotted] table [col sep=comma, x=t, y=x_0_dt_\trajectory] {fig/pendulum/\pca};
        \addplot [myline, color=plotColor2, dashdotted] table [col sep=comma, x=t, y=x_0_dt_\trajectory] {fig/pendulum/\nonlinear};
    \end{axis}
    \begin{axis}[
        myaxis,
        name=ax0,
        ylabel={$\dot{\pendulumPos}_y$},
        at={($ (ax.east) + (15mm,0) $)},
        anchor=west,
        xticklabels={},
    ]
        \addplot [myline, color=plotColor1, solid] table [col sep=comma, x=t, y=x_1_dt_\trajectory] {fig/pendulum/\reference};
        \addplot [myline, color=plotColor4, dashed] table [col sep=comma, x=t, y=x_1_dt_\trajectory] {fig/pendulum/\linear};
        \addplot [myline, color=plotColor3, dotted] table [col sep=comma, x=t, y=x_1_dt_\trajectory] {fig/pendulum/\pca};
        \addplot [myline, color=plotColor2, dashdotted] table [col sep=comma, x=t, y=x_1_dt_\trajectory] {fig/pendulum/\nonlinear};
    \end{axis}

    \begin{axis}[
        myaxis,
        name=ax,
        ylabel={$\ddot{\pendulumPos}_x$},    xlabel={$t$},
        at={($ (ax.south) + (0,-2mm) $)},
        anchor=north,
    ]
        \addplot [myline, color=plotColor1, solid] table [col sep=comma, x=t, y=x_0_ddt_\trajectory] {fig/pendulum/\reference};
        \addplot [myline, color=plotColor4, dashed] table [col sep=comma, x=t, y=x_0_ddt_\trajectory] {fig/pendulum/\linear};
        \addplot [myline, color=plotColor3, dotted] table [col sep=comma, x=t, y=x_0_ddt_\trajectory] {fig/pendulum/\pca};
        \addplot [myline, color=plotColor2, dashdotted] table [col sep=comma, x=t, y=x_0_ddt_\trajectory] {fig/pendulum/\nonlinear};
    \end{axis}
    \begin{axis}[
        myaxis,
        name=ax0,
        ylabel={$\ddot{\pendulumPos}_y$},
        at={($ (ax.east) + (15mm,0) $)},
        anchor=west,    xlabel={$t$},
    ]
        \addplot [myline, color=plotColor1, solid] table [col sep=comma, x=t, y=x_1_ddt_\trajectory] {fig/pendulum/\reference};
        \addplot [myline, color=plotColor4, dashed] table [col sep=comma, x=t, y=x_1_ddt_\trajectory] {fig/pendulum/\linear};
        \addplot [myline, color=plotColor3, dotted] table [col sep=comma, x=t, y=x_1_ddt_\trajectory] {fig/pendulum/\pca};
        \addplot [myline, color=plotColor2, dashdotted] table [col sep=comma, x=t, y=x_1_ddt_\trajectory] {fig/pendulum/\nonlinear};
    \end{axis}
\end{tikzpicture}

        \caption{
            Time evolution of the pendulum's Cartesian coordinates and time derivatives for the reference and corresponding identified models. The identified systems rely on different coordinate transformations of the data. For training, only the first five seconds were used, and the time extrapolation is marked with a red background.
            Please note that the velocities $\dot{\pendulumPos}$ are not derived via differentiation but just a reconstruction from the latent space (as the velocities are part of the considered state vector).
        }
        \label{fig: pendulum results}
    \end{subfigure}
    \caption{Sketch of the mathematical pendulum~(\subref{fig:pendulum}) along with the results of the pH identification~(\subref{fig: pendulum results}).}
\end{figure}

Describing it as a linear pH system leads to significant errors due to the present nonlinearities. By using \aphin{}, on the contrary, it is possible to circumvent this by an identification of generalized coordinates in which the dynamics of the pendulum can be described as a LTI pH system.

\paragraph*{Dataset}
For the application of our approach, we generate training data using a formulation of the dynamics in Cartesian coordinates, resulting in a DAE system with the state vector
\begin{equation}
    \states =
    \left[\begin{array}{cccc}
            \pendulumPos_x & \pendulumPos_y & \dot{\pendulumPos}_x & \dot{\pendulumPos}_y
        \end{array}
        \right]^\transpose \;,
\end{equation}
consisting of the endpoint positions and their time derivatives together with the geometric constraint
\begin{equation}
    \pendulumPos_x^2 + \pendulumPos_y^2 - l^2 = 0 \;,
\end{equation}
which restricts the solution on the manifold of possible states. As for many other mechanical systems, it is more convenient to formulate the dynamics in generalized coordinates. A natural choice for this model is the angle of the pendulum $\pendulumAngle$ and the angular velocity $\pendulumAngVel = \dot{\pendulumAngle}$ resulting in
\begin{equation}
    \latentState = \left[\begin{array}{cc}
            \pendulumAngle & \pendulumAngVel
        \end{array}\right]^\transpose \;
\end{equation} with the corresponding nonlinear ODE system
\begin{equation} \label{eq:pendulum-ode}
    \dot{\pendulumAngle} = \pendulumAngVel \;, \quad
    \dot{\pendulumAngVel} = -\frac{g}{l} \sin \pendulumAngle \;.
\end{equation}

It is then possible to express the Cartesian coordinates only depending on the two generalized coordinates as
\begin{equation}
    \states = \left[\begin{array}{cccc}
            l \sin \pendulumAngle & l \cos \pendulumAngle & l \omega \cos \pendulumAngle & l \omega \sin \pendulumAngle \,
        \end{array}\right]^\transpose \;.
\end{equation}

For the identification of the system, simulations based on varying initial conditions that cover a range of 5\,s serve as training data, while the test trajectories capture 10\,s to investigate time extrapolation. More details on the dataset and the implemented model can be found in \cref{tab:network hyperparameter} again.

\paragraph*{Experiments}
We perform three experiments to showcase the highlights of the proposed framework. Firstly, we directly apply \phin{} (without an autoencoder making a reduction) so that the coordinates can only be linearly transformed through~$\phQ$ and the system is of state dimension~$\stateDim=4$. Secondly, we apply \aphin{} but with the restriction to only use PCA as a \emph{linear} coordinate transformation and reduction to a latent dimension of~$\latentStateDim=2$ and conduct the identification on the reduced coordinates (\aphin{}, linear). Lastly, we apply the identification embedded into a nonlinear autoencoder (\aphin{}, nonlinear).
We then solve each of the identified systems for unseen initial conditions using the implicit midpoint rule as a structure-preserving integration scheme. The resulting solutions for the test trajectory can be seen in \cref{fig: pendulum results}.

\paragraph*{Results}
Obviously, linear coordinate transformations do not result in satisfying approximations as the nonlinearity apparent in the pendulum cannot be captured. Hence, \phin{} alone leads to an identified system in which not both variables are well captured, but only~$\pendulumPos_x$.
Incorporating the system identification into a linear reduction algorithm (\aphin{}, linear), the dynamics are not captured at all, and it becomes apparent that the latent coordinates are not suitable for the identification of linear dynamics.

Therefore, we deploy a nonlinear autoencoder and aim to find a set of coordinates of the same order $\latentStateDim=2$ as the generalized coordinates in which the dynamics appear approximately linear. The results presented in \cref{fig: pendulum results} confirm that this goal is reached as the time evolution for all states and time derivatives follows the reference closely, and even extrapolation is possible due to the harmonic system dynamics. These results are also reflected in the performance measures presented in \cref{tab:perf}. The projection properties \eqref{eq:projection_error} and \eqref{eq:jacobian_error} are satisfied up to machine precision for the linear approach using PCA, but the dynamics are not captured very well for training and test data. In contrast, the nonlinear dimension reduction leads to low mean error values in the latent and state coordinates. The projection properties are fulfilled up to the order $10^{-4}$.
For \aphin{} (nonlinear), the time required for encoding the initial condition, integrating the latent system, and decoding the latent trajectory into the full states accumulates to approximately~$0.08$\,s while training one epoch requires around~$0.18$\,s as stated in \cref{tab: times}.
In general, these results show that our framework is capable of reducing the system dimension to the intrinsic size and identifying linear stable and passive dynamics on the latent coordinates while still being able to reconstruct the physical quantities.

\begin{table}
    \setlength{\tabcolsep}{3pt}
    \centering
    \caption{Performance measures. }
    \label{tab:perf}
    \npdecimalsign{.}
    \nprounddigits{2}
    \begin{tabular}[c]{l l l l l}
        \toprule
         & $\errorProj$                       & $\errorJac$ & $\errorLatentMean$ & $\errorStateMean$ \\
        \midrule
        \textbf{mass-spring-damper}                                                                  \\
        \quad train
         & $-$
         & $-$
         & $\numprint{0.07656346145512051}$
         & $-$                                                                                       \\
        \quad test
         & $-$
         & $-$
         & $\numprint{0.07233414763218916}$
         & $-$                                                                                       \\
        \midrule
        \textbf{pendulum}                                                                            \\
        \ {\phin{} ($\stateDim=4$)}                                                                  \\
        \quad train
         & $-$
         & $-$
         & $\numprint{0.6821843335724542}$
         & $\numprint{1.3181871433951287}$                                                           \\
        \quad test
         & $-$
         & $-$
         & $\numprint{0.6961165654360438}$
         & $\numprint{1.1349923744630812}$                                                           \\
        \ {\aphin{} (linear, $\latentStateDim=2$)}                                                   \\
        \quad train
         & $\numprint{7.0699507e-15}$
         & $\numprint{1.59254e-14}$
         & $\numprint{0.5744512946275222}$
         & $\numprint{0.838349635176911}$                                                            \\
        \quad test
         & $\numprint{6.6907577e-15}$
         & $\numprint{1.59254e-14}$
         & $\numprint{0.7646590521125483}$
         & $\numprint{0.9361470790771349}$                                                           \\
        \ {\aphin{} (nonlinear, $\latentStateDim=2$)}                                                \\
        \quad train
         & $\numprint{7.667193e-4}$
         & $\numprint{1.2583961e-3}$
         & $\numprint{0.06041471647527066}$
         & $\numprint{0.047554838125492174}$                                                         \\
        \quad test
         & $\numprint{2.0384161e-4}$
         & $\numprint{2.4303233e-3}$
         & ${\numprint{0.1624507322454064}}$
         & ${\numprint{0.15388846139013798}}$                                                        \\
        \midrule
        \textbf{disc brake}                                                                          \\
        \ {\aphin{} (nonlinear, $\latentStateDim=3$)}                                                \\
        \quad train
         & $\numprint{1.628e-6}$
         & ${\numprint{5.207e-2}}$
         & ${\numprint{2.91e-4}}$
         & ${\numprint{8.47e-3}}$                                                                    \\
        \quad test
         & $\numprint{1.51e-6}$
         & ${\numprint{3.92e-2}}$
         & ${\numprint{5.44e-4}}$
         & $\numprint{9.04e-3}$                                                                      \\
        \bottomrule
    \end{tabular}
    \npnoround
\end{table}

\subsection{Numerical example III: Disc Brake -- A Coupled Thermo-Mechanical System}

In our final application example, we want to explore the limits of our \aphin{} framework and test its suitability for a (i) high-dimensional, (ii) parameter-dependent, and (iii) input-affine system. Consequently, this time, the autoencoder does not only have to find a suitable nonlinear coordinate transformation, but at the same time (i) significantly reduce the state dimension. In addition, the pH identification network must not only capture autonomous dynamics but also (ii) take into account dependencies on material properties that enter the system matrices and (iii) identify a suitable input matrix.
In particular, we pursue the identification of a low-dimensional pH description for a disc brake, which represents a practical application crucial for the safe operation of vehicles. In this regard, a simulation model of the disc brake can serve as a tool to investigate phenomena leading to undesirable brake squeal caused by deformations due to the frictional heat generated during braking.

The full-order model is implemented as a finite element model in the commercial simulation software \textit{Abaqus} that is based on fully-coupled linear thermoelasticity equations. It is composed of one layer of 60 elements, resulting in a total of $\nNodes = 146$ nodes.
The resulting discretization of the disc brake model is shown in \cref{fig:disc-brake-mesh}, where it can be seen that it is clamped at four nodes which leads to a total number of DOFs of $\stateDim=998$.
Each unconstrained node has $7$ degrees of freedom (DOFs) consisting of a one-dimensional temperature DOF, and 3 displacement and velocity DOFs, respectively. The state vector for the disc brake covering all nodes is consequently defined as
\begin{equation}
    \states =
    \left[\begin{array}{cccc}
            \rT\temps & \rT\disps & \rT{\dot{\disps}}
        \end{array}
        \right]^\transpose \; \label{eq:disc_brake_state}
\end{equation}
with the temperature~$\temps\in\Rdim^{146}$, displacements~$\disps\in\Rdim^{426}$, and velocities~$\dot{\disps}\in\Rdim^{426}$.

The model is a perfect candidate for evaluating the potential of our proposed framework since conventional solution methods using the finite element method require a comparably severe amount of computing time.
Often, there is no access to the operators due to closed-source software, and therefore, there is no possibility to use intrusive reduction methods. The framework presented here offers a solution to still obtain a reduced model based on data.
\begin{figure}[htb]
    \centering
\begingroup%
  \makeatletter%
  \providecommand\color[2][]{%
    \errmessage{(Inkscape) Color is used for the text in Inkscape, but the package 'color.sty' is not loaded}%
    \renewcommand\color[2][]{}%
  }%
  \providecommand\transparent[1]{%
    \errmessage{(Inkscape) Transparency is used (non-zero) for the text in Inkscape, but the package 'transparent.sty' is not loaded}%
    \renewcommand\transparent[1]{}%
  }%
  \providecommand\rotatebox[2]{#2}%
  \newcommand*\fsize{\dimexpr\f@size pt\relax}%
  \newcommand*\lineheight[1]{\fontsize{\fsize}{#1\fsize}\selectfont}%
  \ifx\svgwidth\undefined%
    \setlength{\unitlength}{250.3433075bp}%
    \ifx\svgscale\undefined%
      \relax%
    \else%
      \setlength{\unitlength}{\unitlength * \real{\svgscale}}%
    \fi%
  \else%
    \setlength{\unitlength}{\svgwidth}%
  \fi%
  \global\let\svgwidth\undefined%
  \global\let\svgscale\undefined%
  \makeatother%
  \begin{picture}(1,0.37665271)%
    \lineheight{1}%
    \setlength\tabcolsep{0pt}%
    \put(0,0){\includegraphics[width=\unitlength,page=1]{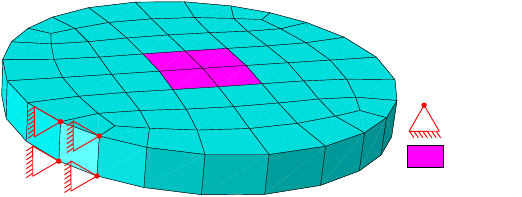}}%
    \put(0.87394111,0.13090745){\makebox(0,0)[lt]{\lineheight{1.25}\smash{\begin{tabular}[t]{l}clamped\end{tabular}}}}%
    \put(0.87141306,0.06597504){\makebox(0,0)[lt]{\lineheight{1.25}\smash{\begin{tabular}[t]{l}heat area\end{tabular}}}}%
  \end{picture}%
\endgroup%

    \caption{Discretization of the disc brake model in \textit{Abaqus} where the heat application area is colored in pink. The model is clamped at the four nodes on the left, marked in red.
    }
    \label{fig:disc-brake-mesh}
\end{figure}

\paragraph*{Dataset}
To evaluate the performance of our proposed method on the disc brake, we consider a scenario in which heat is transferred to the disc through a heat input area consisting of nine nodes (pink area in \cref{fig:disc-brake-mesh}). The heat flux $\heatflux\in [0.01, \; 9.99]\cdot\numprint{e6}\,\text{W/m$^2$}$ is applied as a constant input $\phInputU(\timee)=\heatflux$ on theses nodes. The considered parameter vector $\params=(\heatcond,\density)\in\paramSpace\subset\Rdim^2$ consists of two material properties, the disc's heat conductivity~$\heatcond\in[6.5, \;86.5]\,\text{W/mK}$ and the density~$\density\in[6850,\; 8850]\,\text{kg/m$^3$}$.

Each full-order simulation covers the time interval $\timeSpace = [0, 3]\,$s and the corresponding simulation results are exported with a sampling time of~$1\,$ms resulting in~$\nSamples=3001$ time steps per simulation. The initial conditions are set to $\states_0 = \bm{0}$. For each simulation, quasi-random parameter vectors are sampled using Halton sequences. Details on the dataset are listed in \cref{tab:network hyperparameter}.

The exported simulation results only consist of the displacements~$\disps$ and the temperature~$\temps$. To account for the usual second-order structure of mechanical systems, we further extend the state vector with the velocities $\dot{\disps}$ by using a second-order central differences scheme on the displacements $\disps$.
The complete dataset is given by $\{ \temps(\timee_j; \params_j) \}_{j=1}^{\numData}$, $\{ \disps(\timee_j; \params_j) \}_{j=1}^{\numData}$, $\{ \dot{\disps}(\timee_j; \params_j) \}_{j=1}^{\numData}$, $\{ \phInputUTraj[\timee_j]\}_{j=1}^{\numData}$  and $\{ \params_j \}_{j=1}^{\numData}$ where the data points are calculated from the number of time steps and total simulations as $\numData=\nSamples\nSims$.

Due to the problem's multi-physical nature, each state component is scaled by a constant factor to arrive at the same order of magnitude. Additionally, the training inputs and parameters are normalized component-wise to the interval $[0,1]$ using the maximum and minimum values that occur across all training simulations. The scaling values obtained from the training data are also applied to the test data.
The hyperparameter settings of the used \aphin{} model can be found in \cref{tab:network hyperparameter}.

\paragraph*{Results}
The reduced model can be evaluated in real-time using efficient implementations since the encoder and decoder consist only of rapidly evaluable artificial neural networks, and the pH system in the latent space is a low-dimensional LTI system, which can be solved efficiently using the implicit midpoint rule. The results in \cref{fig:disc_brake_rms_error} show the relative error in the state space where, based on \cref{eq:error_state}, only the correspondent values of each component \cref{eq:disc_brake_state} are picked from the state vector $\states$ and its identified and decoded counterpart $\approximate{\states}$. It can be seen that the framework is able to approximate the heating and expansion behavior of the disc brake very well. The mean test error stays below 4\,\% for the velocity $\dot{\disps}$, below 1\,\% for the temperature $\temps$  and is even less than 0.3\,\% for the displacement $\disps$ component. Moreover, \cref{tab:perf} shows that both the projection error and the error in the Jacobian are very low and thus, the assumptions \eqref{eq:point_projection_property} and \eqref{eq:tangent_projection_property} are weakly fulfilled. Computational-wise, an evaluation of the reduced model requires about~$0.1$\,s as can be seen from \cref{tab: times}.
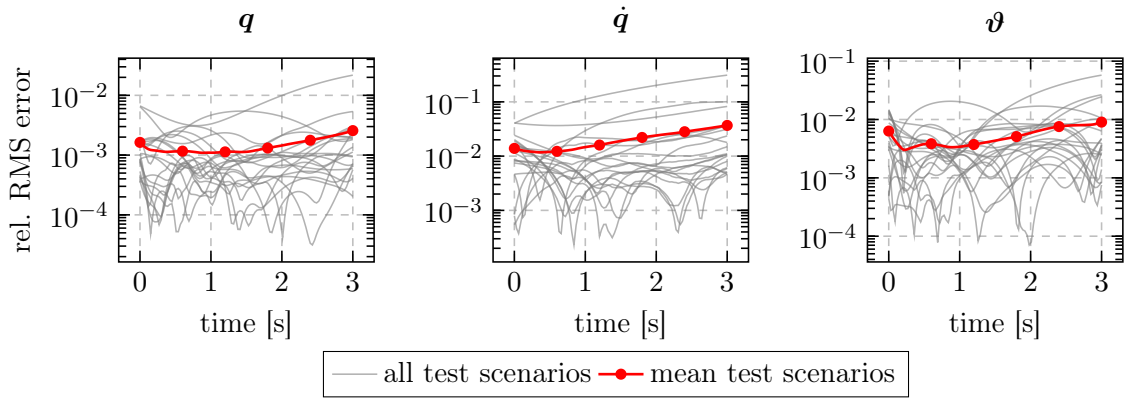
\begin{figure}[htb]
    \centering
    \usetikzlibrary{positioning}

\def\markRepeatValue{60}
\def\numTestSim{19}

\pgfplotsset{
  every axis plot/.append style={line width=1pt},
  every axis/.append style={
    ymajorgrids,
    xmajorgrids,
    grid style={dashed, lightgray,semithick},
   	axis line style = semithick,
   	every tick/.style={semithick,},
  },
}


\pgfplotsset{
    error_scenarios/.style = {
        color = gray, 
        mark = none,
        semithick,
        solid,
        opacity=0.6,
                        },
    mean_error/.style = {
        color = red, 
        mark = *,
        mark repeat = 20,
        mark options={fill=red, scale=0.8},
        solid,
                            },
}

\definecolor{mycolor1}{rgb}{0,0,0}                      
\definecolor{mycolor2}{rgb}{0.85000,0.32500,0.09800}    
\definecolor{mycolor3}{rgb}{0.92900,0.69400,0.12500}    
\definecolor{mycolor4}{rgb}{0.49400,0.18400,0.55600}    
\definecolor{mycolor5}{rgb}{0.46600,0.67400,0.18800}    
\definecolor{mycolor6}{rgb}{0.30100,0.74500,0.93300}    
\definecolor{mycolor7}{rgb}{0.63500,0.07800,0.18400}    

\begin{tikzpicture}
    \begin{axis}
        [
        xlabel = time {[s]},
        ylabel = rel. RMS error,
        title = $\disps$,
        width = 0.3\textwidth,
        at = {(0,0)},
        legend to name = test,
        legend columns = 2,
        ymode =log,
        ]
        \addplot+[
            error_scenarios,
            ] 
            table[
                x=t,
                y = error_state_error_disp_0,
                col sep = comma,
                ]{fig/disc_brake/error/rms_error_state_disp.csv};
        \addlegendentry{all test scenarios}
        \foreach \columnnum in {1,...,\numTestSim}
        {
        \addplot+[
            error_scenarios,
            forget plot, 
            ] 
            table[
                x=t,
                y = error_state_error_disp_\columnnum,
                col sep = comma,
                ]{fig/disc_brake/error/rms_error_state_disp.csv};
        }
        \addplot+[ 
        mean_error,
        ]
        table[
            x = t,           
            y = mean_error_state_error_disp,     
            col sep = comma,    
        ]{fig/disc_brake/error/rms_error_state_disp.csv}; 
        \addlegendentry{mean test scenarios}
    \end{axis}
    \begin{axis}
        [
        xlabel = time {[s]},
        title = $\dot{\disps}$,
        width = 0.3\textwidth,
        at = {(0.3\textwidth,0)},
        name = vel_axis,
        ymode= log,
        ]
        \foreach \columnnum in {0,...,\numTestSim}
        {
        \addplot+[
            error_scenarios,
            ] 
            table[
                x=t,
                y = error_state_error_vel_\columnnum,
                col sep = comma,
                ]{fig/disc_brake/error/rms_error_state_vel.csv};
        }
        \addplot+[ 
        mean_error,
        ]
        table[
            x = t,           
            y = mean_error_state_error_vel,     
            col sep = comma,    
        ]{fig/disc_brake/error/rms_error_state_vel.csv}; 
        \legend{} 
    \end{axis}
    \begin{axis}
        [
        xlabel = time {[s]},
        title = $\temps$,
        width = 0.3\textwidth,
        at = {(0.6\textwidth,0)},
        ymode=log,
        ]
        \foreach \columnnum in {0,...,\numTestSim}
        {
        \addplot+[
            error_scenarios,
            ] 
            table[
                x=t,
                y = error_state_error_temp_\columnnum,
                col sep = comma,
                ]{fig/disc_brake/error/rms_error_state_temp.csv};
        }
        \addlegendentry{error scenarios}
        \addplot+[ 
        mean_error,
        ]
        table[
            x = t,           
            y = mean_error_state_error_temp,     
            col sep = comma,    
        ]{fig/disc_brake/error/rms_error_state_temp.csv}; 
        \addlegendentry{mean error}
        \legend{} 
    \end{axis}
    \node[
        below = 1.05cm of vel_axis.south,
        align=left,
        anchor=north,
        ] {%
    \ref*{test}
    };
\end{tikzpicture}%
    \caption{Relative state error \eqref{eq:error_state} divided into the different components for test data.}
    \label{fig:disc_brake_rms_error}%
\end{figure}

In addition to the error metrics globally defined for the disc brake, we also want to report local errors occurring at the individual nodes of the model. For this, we introduce error measures for each domain on node-level as
\begin{align}
    e_{{\temp}}(\temp_{j,k},\approximate{\temp}_{j,k}) = |{\temp_{j,k} -  \approximate{\temp}_{j,k}}|, \
    e_{{\disps}}(\disps_{j,k},\approximate{\disps}_{j,k}) = \norm{\disps_{j,k} -  \approximate{\disps}_{j,k}}_{2}, \text{ and }
    e_{\dot{\disps}}(\dot{\disps}_{j,k},\approximate{\dot{\disps}}_{j,k}) = \norm{\dot{\disps}_{j,k} -  \approximate{\dot{\disps}}_{j,k}}_{2}
    \label{eq:error_temp}
\end{align}
where $\temp_{j,k}:=\temp_k(\timee_j; \params_j),\, \disps_{j,k}:=\disps_k(\timee_j; \params_j),\, \dot{\disps}_{j,k}:=\dot{\disps}_k(\timee_j; \params_j)$ define the value at node $k=1,\dots,\nNodes$ and time sample $j=1,\dots,\nSamples\nSims$ for the temperature, displacements and velocities, respectively. Their approximations are derived analogously. These measures give us absolute error values for the difference between the reference temperature and its approximation, as well as the Euclidean distance between the reference and approximated simulation for the displacements and velocities. For the evaluation, we provide a visual comparison between some sample test simulations obtained from the FE software and the approximations using our \aphin{} method in \cref{fig:disc_brake_3d}. As can be seen there, the largest errors for the temperature and displacements occur around the heating area, while the velocity error is the largest in areas on the opposite side of the clamping.
Nevertheless, the behavior of the disc brake is captured well for a variety of different parameters and inputs.

\begin{figure}[htb]
    \centering
    \input{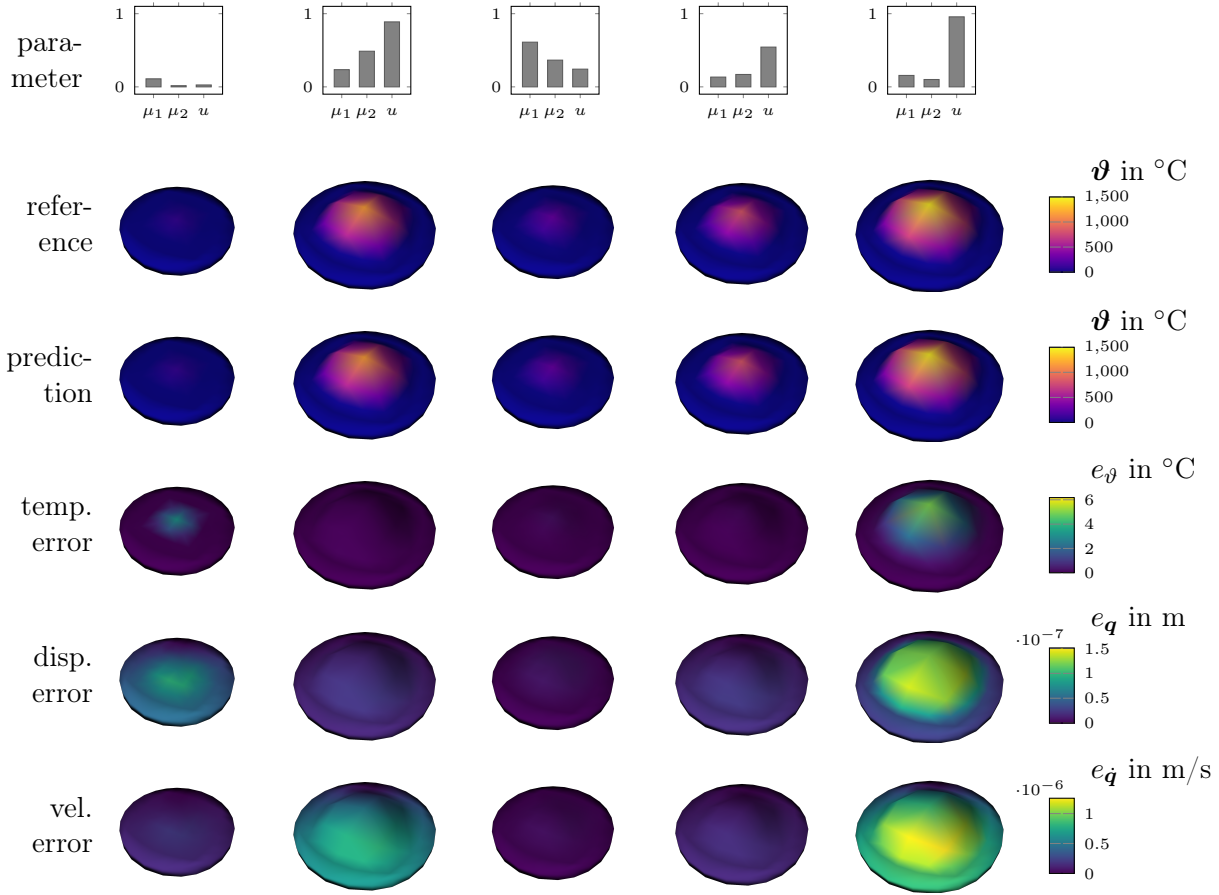}%
    \caption{Visualization of the displacements and temperature of the disc brake for five example test simulations comparing the reference results (1st row) with their approximations (2nd row). The corresponding normalized simulation parameters and inputs are indicated on top of the disc brake visualizations. In addition to the simulation results itself, the approximation quality for the temperature, displacement, and velocity field are reported as colored discs with respect to the corresponding error in rows three to five. For visualization purposes, the displacements are scaled by 400 and the last time point of each simulation is shown.}
    \label{fig:disc_brake_3d}%
\end{figure}

\section{Conclusion}

In this paper, we propose a novel framework for structure-enforcing system identification that operates in the port-Hamiltonian (pH) formulation. Low-dimensional linear pH systems are identified solely from data of parameter-dependent, input-affine, and (potentially) high-dimensional systems.
The corresponding (parameter-dependent) pH system and input matrices are identified by a neural network, which we refer to as pH identification network~(\phin).
The architecture of the framework naturally inherits the important system theoretical properties of passivity and stability of the identified system in the latent pH space. Simultaneously, the description as a pH system makes the identified model tailor-made for multi-physics problems.

For data created by high-dimensional or very nonlinear systems, we encapsulate the \phin{} in a nonlinear coordinate transformation by means of an autoencoder resulting in \aphin{}. The autoencoder transforms the original system states into (low-dimensional) coordinates suitable for describing their dynamics as a linear pH system. Furthermore, under (mild) assumptions, the properties of the pH system are preserved not only in the latent but also in the state space.

We were able to demonstrate the capabilities of the approach by various numerical examples. It is possible to accurately reconstruct the system matrices for a system whose description as pH formulation is known. Furthermore, the nonlinear dynamics of a pendulum are captured with an identified linear pH system through nonlinear coordinate transformations.
In addition, an application-oriented thermomechanical model in the form of a high-dimensional disc brake served as an example of a parameter-dependent and input-affine system for which it was nevertheless possible to efficiently model its dynamic behavior.

The proposed framework's nature limits it to data-generating dynamical systems that are stable and passive. Although this is intentional, it is expected to result in a poor fit for other types of systems. Furthermore, while the identified system in the latent space may be suitable for extrapolation to unseen scenarios, the surrounding autoencoder is prone to errors outside of the training regime.
Another limitation of the method is that it relies on complete measurements/high-fidelity data and therefore the accuracy is highly dependent on the quality of the data.

Hence, an exciting direction for future work is to extend the data-driven system identification of pH systems to more realistic situations where only partial observations are available. In addition, constrained autoencoders, which by definition fulfill projection properties, can be introduced into the framework to satisfy the assumptions necessary to transfer the pH properties to physical space.
Moreover, a similar approach can be used to learn nonlinear pH systems in the latent space. This case is already covered in our analysis from \Cref{sec:systheo_properties_preserved}, which discusses transferring system-theoretical properties from $\phSpace$ to $\stateSpaceApprox$. Furthermore, the same \aphin{} framework can also be used to identify a lifted higher-dimensional linear system in the latent space, which could be particularly advantageous for some nonlinear examples. In addition, the ability to identify input-affine systems could make the framework especially suitable for control scenarios where a complex system needs to be efficiently controlled.

\appendix
\appendixpage

\section{Proof Boundedness of Solutions}
\label{App Proof Boundedness of Solutions}
\begin{proof}
    The proof is a generalization of \cite[Thm~3.9]{Peng2016} from Hamiltonian to pH systems and works via a contradiction.
    Assume there exists $t \in \timeSpace$ such that $\phStateTraj[t] \notin \boundedNeighbourhood$.
    Since the initial value $\phState_0$ is by Assumption~(i) contained in $\boundedNeighbourhood$ and the solution is continuous in $t$, there exists a time $\timeBoundary \in \timeSpace$ such that $\phStateTraj[\timeBoundary] \in \partial \boundedNeighbourhood$.
    But then, it holds by Assumption~(ii) that
    \begin{align*}
        \underbrace{
            \Hamiltonian(\phStateTraj[\timeBoundary])
        }_{
            > \Hamiltonian(\phStateTraj[t_0])
        }
        - \Hamiltonian(\phStateTraj[t_0])
        > 0,
    \end{align*}
    which is with Assumption (iii) a contradiction to the integral over the dissipation inequality
    \begin{align*}
        \Hamiltonian(\phStateTraj[\timeBoundary]) - \Hamiltonian(\phStateTraj[t_0])
        \leq \int_{t_0}^{\timeBoundary} \rT{\phOutputTraj[t]} \underbrace{\phInputUTraj[t]}_{=0} \rmd t = 0.
    \end{align*}
    Thus, $\phStateTraj[t] \in \boundedNeighbourhood$ for all $t \in \timeSpace$.
    Moreover, the solutions are guaranteed to stay in a ball, i.e., $\norm{\phStateTraj[t] - \phState_0}_{\latentStateSpace} \leq \constBoundedSolutions$ for all $\timee \in \timeSpace$, with the radius $\constBoundedSolutions = \sup_{\phState_1 \in \boundedNeighbourhood} \norm{\phState_1 - \phState_0}_{\latentStateSpace}$.
\end{proof}

\setcounter{table}{0}
\renewcommand{\thetable}{A\arabic{table}}

\section{Network Implementation and Runtimes}
\label{App1}

\begin{table}[htb]
    \setlength{\tabcolsep}{3pt}
    \centering
    \caption{Network hyperparameter}
    \label{tab:network hyperparameter}
    \npdecimalsign{.}
    \nprounddigits{2}
    \begin{tabular}[c]{l l l l}
        \toprule
                                                           & \textbf{mass-spring-damper}    & \textbf{pendulum}          & \textbf{disc brake}        \\
        \midrule
        \textbf{latent dim.} $\latentStateDim$             & 6                              & 2                          & 3                          \\
        \textbf{intermediate dim.} $\intermediateStateDim$ & -                              & -                          & 8                          \\
        \textbf{autoencoder layer sizes}                   & -                              & [32, 32 ,32]               & [64, 32, 16]               \\
        \textbf{$\neuralnetWeights$ layer sizes}           & [16, 16, 16, 16]               & -                          & [32, 32, 32]               \\
        \textbf{activation function}                       & elu                            & elu                        & elu                        \\
        \textbf{loss factors}                                                                                                                         \\
        $\quad \lossfactor_{\text{rec}}$                   & -                              & 1                          & 1                          \\
        $\quad \lossfactor_{\text{pH}}$                    & 1                              & 0.1                        & 0.1                        \\
        $\quad \lossfactor_{\text{con}}$                   & -                              & 0.001                      & 0.001                      \\
        $\quad \lossfactor_{\text{L1}}$                    & \npnoround\numprint{1e-07}     & \npnoround\numprint{1e-10} & \npnoround\numprint{1e-09} \\
        \textbf{early stopping}                            & \checkmark                     & \checkmark                 & \checkmark                 \\
        \textbf{epochs}                                    & 2000                           & 4000                       & 2000                       \\
        \textbf{batch size}                                & 64                             & 256                        & 256                        \\
        \textbf{training data shape} $[\nSims \times \nSamples \times \stateDim]$
                                                           & $[90 \times 400 \times 6]$
                                                           & $[12 \times 500 \times 4]$
                                                           & $[48 \times 3001 \times 1022]$                                                           \\
        \textbf{test data shape} $[\nSims \times \nSamples \times \stateDim]$
                                                           & $[30 \times 400 \times 6]$
                                                           & $[6 \times 1000 \times 4]$
                                                           & $[20 \times 3001 \times 1022]$                                                           \\
        \bottomrule
    \end{tabular}
    \npnoround
\end{table}

\begin{table}[htb]
    \setlength{\tabcolsep}{3pt}
    \centering
    \caption{Training and evaluation times on a Apple M1 Max with a 10-Core CPU, 24-Core GPU, and 64 GB of RAM.}
    \label{tab: times}
    \npdecimalsign{.}
    \nprounddigits{3}
    \begin{tabular}[c]{l l l l}
        \toprule
                                              & \textbf{mass-spring-damper} & \textbf{pendulum}             & \textbf{disc brake}    \\
        \midrule
        \textbf{training time per epoch in} s & \numprint{0.3153211}        & \numprint{0.182887461}        & \numprint{2.84521138}  \\
        \textbf{evaluation time per run in} s & \numprint{0.002764201}      & \numprint{0.0775530338287353} & \numprint{0.106170893} \\
        \bottomrule
    \end{tabular}
    \npnoround
\end{table}

\acknowledgements{
    Funded by Deutsche Forschungsgemeinschaft (DFG, German Research Foundation) under Germany's Excellence Strategy - EXC 2075 - 390740016. We acknowledge the support by the Stuttgart Center for Simulation Science (SimTech). This research was partially funded by the Ministry of Science, Research, and the Arts (MWK) Baden-Württemberg, Germany, within the Artificial Intelligence Software Academy (AISA) and the InnovationsCampus Future Mobility.
}

\printbibliography





\end{document}